\documentclass[11pt,a4paper]{article}
\usepackage{fullpage}
\usepackage{authblk}

\usepackage{mathrsfs}
\usepackage{amsfonts}
\usepackage{amssymb}
\usepackage{bm}
\usepackage{multirow,bigstrut}
\usepackage{threeparttable}
\usepackage{enumerate}
\usepackage{tikz}
\usepackage{graphicx}
\usepackage{stfloats}
\usepackage{array}
\usepackage{amsmath}
\usepackage[colorlinks,
            linkcolor=red,
            anchorcolor=blue,
            citecolor=green
            ]{hyperref}

\newenvironment{proof}{\noindent{\em \textbf{Proof.}}}{\quad \hfill$\Box$\vspace{2ex}}

\newtheorem{theorem}{Theorem}[section]

\newtheorem{lemma}[theorem]{Lemma}

\numberwithin{equation}{section}

\def \T {\mathbb{T}}

\def\i { \bm {i}}

\def\sgn{\mbox{sgn}}

\newcommand{\norm}[1]{\left\lVert#1\right\rVert}
\newcommand{\abs}[1]{\left|#1\right|}

\title{Multichannel reconstruction from nonuniform samples with application to image recovery}
	\author{Dong Cheng\thanks{chengdong720@163.com} }
	
	\author{Kit Ian Kou\thanks{kikou@umac.mo}}
	
	\affil{\normalsize{Department of Mathematics, Faculty of Science and Technology, University of Macau, Macao, China}}
	
	\date{}

\begin{document}
  \maketitle
\begin{abstract}
\normalsize
The multichannel trigonometric reconstruction from uniform samples was proposed recently.  It not only makes use of multichannel information about the signal but is also capable to generate various kinds of interpolation formulas according to the types and amounts of the collected samples. The paper presents the theory of multichannel interpolation from nonuniform samples.  Two distinct models of nonuniform sampling patterns are considered, namely recurrent and generic nonuniform sampling. Each model involves two types of samples: nonuniform samples of the observed signal and its derivatives.  Numerical examples and quantitative error analysis are provided to demonstrate the effectiveness of the proposed algorithms. Additionally, the proposed algorithm for recovering highly corrupted images is also investigated. In comparison with the median filter and correction operation treatment, our approach  produces superior results with lower errors.
\end{abstract}


\begin{msc}
 42A15,  94A12,	65T50, 94A08.
\end{msc}

%\linenumbers
%=================================================================
\section{Introduction}\label{S1}
Sampling and reconstruction are used as fundamental tools in data processing and communication systems.
Classical uniform sampling \cite{zayed1993advances} is an effective tool to recover signals and has been applied in many applications.
However there are various instances that reconstruction of signals from their nonuniform
samples are required, such as computed tomography \cite{sidky2008image}, magnetic resonance \cite{liang2000principles} and radio astronomy \cite{burke2009introduction}.
Numerous approaches have been proposed in the literature to reconstruct bandlimited signals from nonuniform samples
\cite{yen1956nonuniform,yao1967some,jerri1977shannon,feichtinger1992irregular}. A widely known nonuniform sampling theorem \cite{zayed1993advances,seip1987an} may be stated as follows. Let $\{t_n\}_{n\in \mathbb{Z}}$ be a sequence of real numbers such that $\abs{t_n-\frac{n\pi}{\sigma}}<\frac{\pi}{4\sigma}$, then a $\sigma$-bandlimited function can be reconstructed by
\begin{equation}\label{Rnonuniform}
f(t)=\sum_{n=-\infty}^{\infty}f(t_n)S_n(t)
\end{equation}
where
\begin{align*}
	S_n(t)=\frac{G(t)}{G'(t_n)(t-t_n)},\quad  G(t)=(t-t_0)\prod_{k=1}^{\infty}\left(1-\frac{t}{t_k} \right) \left(1-\frac{t}{t_{-k}} \right)
\end{align*}
and  (\ref{Rnonuniform}) converges uniformly on any compact subset of $\mathbb{R}$. The series
(\ref{Rnonuniform})  is an extension of  Lagrange interpolation. Unlike the uniform sampling case,  (\ref{Rnonuniform})  contains infinitely many terms and the interpolating functions $S_n(t)$ involve  complicated components. These factors bring difficulties for exact reconstruction of a bandlimited function from nonuniform samples. To give a simpler approximation for reconstructing a bandlimited function from nonuniform samples, the authors in \cite{Maymon2011sinc} proposed a new kind of sinc interpolation method and they restricted $S_n(t)$ in (\ref{Rnonuniform}) to be of the   form $\mathrm{sinc}[\sigma(t-\tilde{t}_n)]$, %${\sin \sigma(t-\tilde{t}_n)}/{\sigma(t-\tilde{t}_n)} $
where $\tilde{t}_n=nT+ \zeta_n$ and $\zeta_n$ is a sequence of random variables independent of $G(t)$. %sinc function is defined by $\frac{\sin t}{t}$.
This restriction guarantees that the interpolating functions only consist of translation of sinc function, just like most cases of uniform interpolation \cite{higgins2000sampling,liu2010new,cheng2017novel,chengkou2018generalized}.  However, the restriction strategy simplifies reconstruction problem but introduces error inevitably.  To overcome the error, several methods for determining suitable $\tilde{t}_n$
%which is not necessary to be $t_n$
were analyzed \cite{Maymon2011sinc}. By the similar idea, the sinc interpolation for nonuniform samples in fractional Fourier domain was studied in \cite{XU2016311}.

In a real application, there are only finitely many samples, albeit with large amount, are given in a bounded region. Interpolating by sinc functions or any other bandlimited functions has some limitations. On the one hand, the bandlimited interpolating functions cannot be time limited by the uncertainty principle
%,and must be truncated in practice
, thereby the approximation error is introduced.  On the other hand, the interpolating functions $S_n(t)$ for nonuniform samples are complicated and there is no closed form in general. Therefore, treating a finite amount of data as samples of a periodic function is  a convenient and  feasible choice \cite{margolis2008nonuniform}. As we know, sines and cosines are classical periodic functions, they have wide applications in modern science. It is no exaggeration to say that trigonometry pervades the area of signal processing. We know that $\{e^{\i nt}: n\in \mathbb{Z}\}$ is an orthogonal system and is complete in square integrable functions space on unit circle $\mathbb{T}$, i.e., ${L}^2(\mathbb{T})$. Besides, these   functions possess  elegant symmetries
and concise  frequency  meanings. These desirable properties of trigonometric functions could make interpolation much simpler and more effective \cite{margolis2008nonuniform,NAVASCUES2018}. In fact, in the early 1841,  Cauchy first proved a sampling interpolation theorem on trigonometric polynomials \cite{cauchy1841memoire}. It may be recognized as the headstream of sampling theory \cite{benedetto2012modern}. Cauchy's result states that  if $f(t)=\sum_{\abs{n}\leq M} c_ne^{2\pi \i tn}$, then it can be written as a sum of its sampled values $f(\frac{k}{2M+1})$,   $ 0\leq k\leq 2M$, each multiplied by a interpolating function. That is,
\begin{equation*}
	f(t) =\frac{1}{2M+1}\sum_{k=0}^{2M} f(\frac{k}{2M+1})\frac{(-1)^k \sin \pi (2M+1) t}{\sin \pi(t-\frac{k}{2M+1})}.
\end{equation*}
Certain studies  have been given to the problem of interpolating finite length samples by trigonometric functions or discrete Fourier transform.
In a series of papers \cite{schanze1995sinc,candocia1998comments,dooley2000notes},    the sinc interpolation of discrete periodic signals were extensively discussed. Although   referred to as sinc interpolation, the resulting interpolating functions are trigonometric. In \cite{jacob2002sampling}, the authors decomposed a periodic signal in a basis of shifted and scaled versions of a generating function. Moreover,  an error analysis for the  approximation method was also addressed.  A generalized trigonometric interpolation was considered in \cite{NAVASCUES2018} to make a good approximation for non-smooth functions.   Recently, the nonuniform sampling theorems for trigonometric polynomials were   presented \cite{margolis2008nonuniform,xiao2013sampling}. Selva \cite{selva2015fft} proposed a FFT-based interpolation of nonuniform samples. However, this method  is valid  only for nonuniform samples lying in a regular grid rather than  for  non-uniformly distributed data in  the general sense.

In all of above mentioned interpolation methods for finite length discrete points, only the samples of original function are processed. As an extension of trigonometric interpolation,
a multichannel interpolation of finite length samples was suggested in \cite{chengkou2018multi1}.
This novel method  makes good use of multifaceted information (such as derivatives, Hilbert transform) of function  and is capable of generating various useful interpolation formulas by selecting suitable parameters according to   the types and  amount of collected data.
In addition, it can be used to approximate some   integral transformations (such as Hilbert transform). A fast algorithm based on FFT makes multichannel interpolation more effective and stable. However, only the cases of uniform sampling were considered in \cite{chengkou2018multi1}.  There is a need to extend multichannel interpolation such that non-uniformly distributed data can be processed.

The purpose of this paper is  to establish the  theory of multichannel interpolation for
non-uniformly distributed data.  We will consider  two kinds of nonuniform sampling patterns: recurrent and generic nonuniform sampling. Meanwhile, each kind of  nonuniform sampling
involves two types of samples: its own nonuniform samples and derivative's samples.
%  We will propose two kinds of sampling-interpolation theorems. The first is about the reconstruction of a function by  interpolating   its own non-uniformly distributed  samples. The second is about the reconstruction of a function by interpolating its samples and  the samples of  its derivatives which are non-uniformly distributed. In detail, each kind   theorem contains two types of sampling patterns: recurrent nonuniform sampling and generic nonuniform sampling.
There are four nonuniform  interpolation formulas will be analyzed. All closed-form expressions of interpolating functions are derived. Some examples are also demonstrated. We show that the trigonometric polynomial (also called periodic bandlimited function) of finite order can be exactly reconstructed by the proposed interpolation formulas provided that the total number of samples is enough.  Error analysis of the reconstruction for non-bandlimited square integrable functions are analyzed. Concretely,  the contributions of this paper may be summarized as follows:
\begin{enumerate}
	\item  We propose four types of interpolation formulas for non-uniformly distributed data. The proposed formulas involves not only samples of   $f$ but also samples of  $f'$, where $f$ is the function to be reconstructed. If the given data is sampled from a periodic bandlimited function, then we   arrive at a perfect reconstruction provided that the amount of data is larger than the bandwidth.
	\item  We analyze the error that arise in  reconstructing a non-bandlimited function by the proposed formulas. In particular, a comparison of performance on reconstructing    square integrable functions (not necessarily to be bandlimited) by these  formulas is made.
	\item  Applying the proposed interpolation formulas, we develop   algorithms for the recovery of damaged pixels which are non-uniformly located in  a degraded image.
	%The reconstructed image  preserves the gradient of   original image, due to the introduction of derivative in the proposed algorithm. Besides,
	The algorithms   perform  well and can be efficiently implemented. Thus they could be   good pre-processing methods for some more sophisticated   approaches (such as deep learning) in the  image recovery problem.
\end{enumerate}

This paper is organized as follows.  In Section \ref{S2} some preparatory knowledge of Fourier series and multichannel interpolation are reviewed. Section \ref{S3} and \ref{S4} formulate four types of interpolation formulas for non-uniformly distributed data. The numerical examples and error analysis are presented in Section \ref{S5}. The application of proposed interpolation method to image recovery is shown in Section \ref{S6}. Finally, conclusion will be drawn in Section \ref{S7}.

\section{Preliminaries}\label{S2}

\subsection{Fourier series}

Without loss of generality, we will consider the functions defined on unit circle $\mathbb{T}$. Let $L^2(\mathbb{T})$ be the totality of square integral functions defined on $\mathbb{T}$. It is known that
$L^2(\mathbb{T})$ is a Hilbert space embedded with the inner product
\begin{equation*}
	(f,h):= \frac{1}{2\pi} \int_{\T}f(t)\overline{h(t)}dt,~ ~~\forall f, h \in L^2(\T).
\end{equation*}
For $f\in L^2(\mathbb{T})$, it can be expanded as
\begin{equation*}%\label{function}
	f(t)=\sum_{n\in \mathbb{Z}} a(n) e^{\i nt}
\end{equation*}
where the Fourier series is convergent to $f$ in $L^2$ norm. The general version of Parseval's identity is of the form
\begin{equation*}
	(f,h) =\sum_{n\in \mathbb{Z}}a(n) \overline{b(n)},
\end{equation*}
where $\{a(n)\}$ and $\{b(n)\}$ are Fourier coefficients of $f$ and $h$ respectively. The convolution theorem  manifests as
\begin{equation*}%\label{conv}
	(f*h)(t)  :=\frac{1}{2\pi} \int_{\T}f(s)h(t-s)ds=\sum_{n\in \mathbb{Z}}a(n) b(n) e^{\i nt}.
\end{equation*}

The circular Hilbert transform \cite{king2009hilbert,mo2015afd}  is an useful tool in harmonic analysis and signal processing. It is defined by the singular integral
\begin{equation*}
	\mathcal{H}f(t) :=\frac{1}{2\pi}\mathrm{p.v.}\int_{\T}
	f(s)\cot\left(\frac{t-s}{2}\right)ds  =\sum_{n\in \mathbb{Z}}(-\i\sgn(n))a(n) e^{\i n t}
\end{equation*}
where $\sgn$ is the signum function taking values $1$, $-1$ or $0$ for $n>0$, $n<0$ or $n=0$ respectively.
From the definition, we see that it is   simple and straightforward  to compute Hilbert transform for trigonometric functions. Thus trigonometry-based interpolation  can be availably used to approximate Hilbert transform as well.
%This is a congenital advantage of the proposed method, which  most of the other methods may not possess.

%================================================
\subsection{Multichannel interpolation}

Multichannel interpolation proposed in \cite{chengkou2018multi1} is about the reconstruction problem of finite order  trigonometric polynomials. To maintain consistent terminology with the classical case, in what follows, a finite order  trigonometric polynomial is called a periodic bandlimited function, or briefly a bandlimited function.  Let $\mathbf{N}=(N_1,N_2)\in \mathbb{Z}^2$, in the sequel we denote by  $B_{\mathbf{N}}$ the totality of bandlimited functions with the following form:
\begin{equation*}%\label{bandsignal}
	f(t)=\sum_{n\in I^{\mathbf{N}}}a(n)e^{\i nt}  ,~~~I^{\mathbf{N}}=\{n: N_1\leq n \leq N_2\}.
\end{equation*}
The bandwidth of $f$ is defined by   the  cardinality of $I^{\mathbf{N}}$, denoted by $\mu(I^{\mathbf{N}})$.

For $1\leq m \leq M$, let
\begin{align}
	h_m(t)  &=\sum_{n\in \mathbb{Z}}b_{m}(n)e^{\i  nt}, \label{sysfunction}  \\
	g_m(t)   & = (f * h_m)(t) =  \frac{1}{2\pi} \int_{\T}f(s)h_m(t-s)ds.\nonumber
\end{align}
Suppose that $\frac{N_2-N_1+1}{M}=K\in \mathbb{N}^{+}$, we cut $I^{\mathbf{N}}$ into pieces as
$ I^{\mathbf{N}}=\bigcup_{j=1}^M I_j$, where
\begin{equation*}
	I_j=\{n: N_1+(j-1)K\leq n\leq N_1+jK-1\}.
\end{equation*}
The multichannel interpolation indicates that
a bandlimited function $f\in B_{\mathbf{N}}$ can be reconstructed by samples of $g_m$, namely,
\begin{equation}\label{drictexpression}
f(t)=  \frac{1}{K} \sum_{m=1}^{M} \sum_{p=0}^{K-1}g_m(\frac{2\pi p}{K}) y_m(t-\frac{2\pi p}{K})
\end{equation}
provided that $M\times M$ matrix $\mathbf{H}_n =\left[ b_m(n+jK-K)\right]_{jm}$  is invertible for every  $n\in I_1$. Here, the interpolating functions are constructed by the elements of $\mathbf{H}_n^{-1}$.  We denote the inverse matrix as
\begin{equation*}
	\mathbf{H}_n^{-1}= \begin{bmatrix}
		q_{11} (n) & q_{12} (n)&\cdots &q_{1M}(n)\\
		q_{21}(n) & q_{22}(n) &\cdots &q_{2M} (n)\\
		\vdots&\vdots& ~ & \vdots \\
		q_{M1} (n)& q_{M2}(n) &\cdots &q_{MM}(n)
	\end{bmatrix}.
\end{equation*}
The interpolating function $y_m$ for $1\leq m \leq M$ is given by
\begin{equation*}%\label{ym}
	y_m(t)= \sum_{n\in I^{\mathbf{N}}}r_{m}(n)e^{\i n t}
\end{equation*}
where
\begin{equation*}%\label{defrm}
	r_{m}(n)= \begin{cases}
		q_{mj}(n+K-jK),&  \text{if}~n\in I_j, ~j=1,2,\cdots,M ,  \\
		0 & \text{if}~ n\notin I^{\mathbf{N}}.
	\end{cases}
\end{equation*}

In the following sections, using the powerful technique of multichannel interpolation, we  present four types of nonuniform interpolation formulas. Since there are some similar concepts involved in the following parts,  several notations may appear repeatedly with minor difference. We particularly remark that a notation could have different meanings across different parts.

\section{Multichannel  interpolation of  recurrent non-uniformly distributed data}\label{S3}

The recurrent nonuniform sampling often arises in time-interleaved analog-to digital converting process
\cite{strohmer2006fast,Maymon2011sinc}. As for recurrent nonuniform sampling, a classical result that   has to be mentioned  is the Papoulis' generalized sampling expansion (GSE) \cite{papoulis1977generalized}. The differences between GSE and the multichannel interpolation are mainly as follows:
\begin{itemize}
	\item   The GSE involves infinite summation and is applied to recovering functions defined on whole real line. Therefore the truncation is inevitable in practice. The multichannel interpolation is about reconstructing a finite length function from a finite number of samples.
	\item  There is a FFT-based fast algorithm to  implement the multichannel interpolation. The implementation of GSE is more complicated.
	\item   The multichannel interpolation can be extended to the generic nonuniform sampling case (see Section \ref{S4}). However, to the authors' knowledge, there is no generic nonuniform sampling formula based on GSE.
\end{itemize}

In the following two subsections, based on multichannel interpolation technique, we derive two interpolation formulas associated with recurrent nonuniform samples: one  concerns  derivative of function and the other does not.
Throughout Section \ref{S3}, let $m_0\in \mathbb{N}^+$ and $t_p=\frac{2\pi p}{m_0}$ for $p=0,1,\dots,m_0-1$.

\subsection{Recurrent nonuniform samples}

%In this subsection, we set $t_p=\frac{2\pi p}{m_0}$ for $p=0,1,\dots,m_0-1$ and %$0<\alpha<\frac{2\pi}{m_0}$.
By setting $b_1(n)=1$, $b_2(n)=e^{\i n \alpha}$ with $0<\alpha<\frac{2\pi}{m_0}$ in (\ref{sysfunction}) and applying multichannel interpolation, it is easy to have the interpolation formula for  recurrent non-uniformly distributed data:
\begin{equation}\label{recurrent_nonuniform1}
\mathcal{T}_1(f,2m_0,\alpha,t) = \sum_{p=0}^{m_0-1}f( t_p)y_{1,\alpha}(t-t_p) + f(\alpha +t_p)y_{2,\alpha}(t-t_p) .
\end{equation}
The resulting interpolating functions are:
\begin{align}
	y_{1,\alpha}(t) & = \frac{\left(e^{\i m_0 t}-1\right) \left(e^{\i ( m_0 \alpha + {N_1} t)}-e^{\i  (m_0+{N_1})t}\right)}{m_0\left(e^{\i t}-1\right) \left(e^{\i  m_0  \alpha}-1\right)},\label{intepf1}\\
	y_{2,\alpha}(t)& =\frac{e^{\i N_1 t}\left(e^{\i m_0 t}-1\right) \left(e^{\i m_0 \alpha  }-e^{\i m_0 t}\right) e^{\i (1-m_0-N_1)\alpha}}{m_0\left(e^{\i m_0 \alpha }-1\right) \left(e^{\i \alpha }-e^{\i t}\right)}. \label{intepf2}
\end{align}
\begin{figure*}[!t]
	\centering
	% Requires \usepackage{graphicx}
	\includegraphics[width=4.5in]{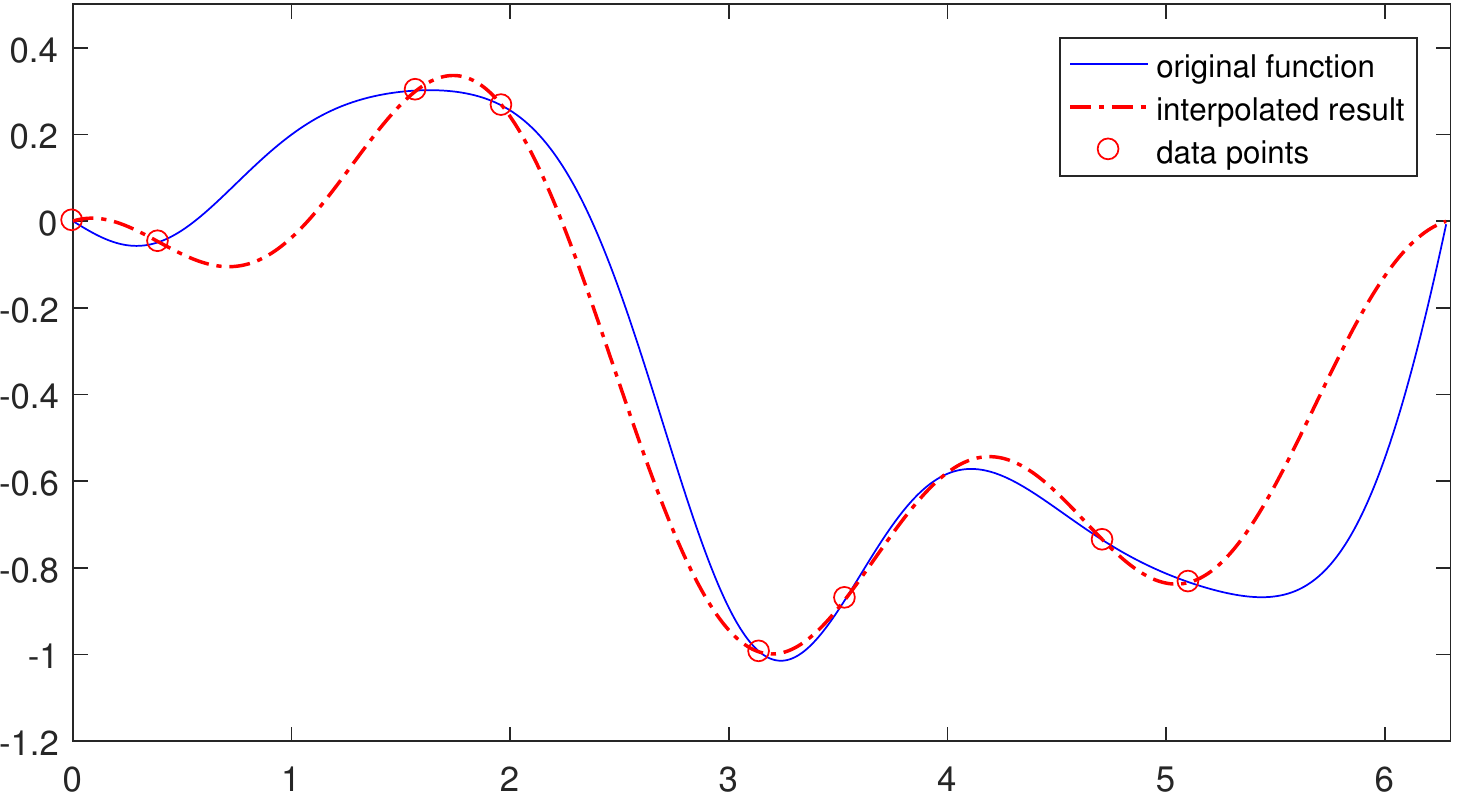}\\
	\caption{ Illustration of   interpolation for   recurrent nonuniform samples and its  consistency. The blue line is  original function.  The red  dash-dot line  is the interpolated result for the given data points.}\label{figinterlaced11}
\end{figure*}
It is noted that for the case $\alpha=\frac{\pi}{m_0}$, the formula (\ref{recurrent_nonuniform1}) reduces to the uniform sampling interpolation. Another fact is that if $m_0$ is larger than the half bandwidth of $f$, then
the reconstruction is exact. Most often, one may have no need to compute the interpolating functions, since $y_{i,\alpha}(t)$ in  (\ref{intepf1}) and    (\ref{intepf2}) can be implemented by FFT efficiently \cite{chengkou2018multi1}.

Importantly, the interpolation  consistency holds for  the formula (\ref{recurrent_nonuniform1}). Namely, the following identities hold:
\begin{align}
	y_{1,\alpha}(t_q-t_p)& = \delta_{pq}, ~y_{1,\alpha}(t_q-t_p+\alpha) = 0 \label{consist3}\\
	y_{2,\alpha}(t_q-t_p)& = 0, ~y_{1,\alpha}(t_q-t_p+\alpha) = \delta_{pq}  \label{consist4}
\end{align}
where $p,q=1,2,\dots,m_0$ and $\delta_{pq} =1$ for $p=q$ and $\delta_{pq} =0$  otherwise.   By direct computation from interpolating functions (\ref{intepf1}) and (\ref{intepf2}), formulas (\ref{consist3}) and (\ref{consist4}) hold. A concrete example is depicted in Figure \ref{figinterlaced11}. Here, $m_0=4,~ \alpha=\frac{\pi}{2m_0}$ and the original function is given by
\begin{equation}\label{ex-signal}
f(t)=0.05t(t-2\pi)(0.04t^2 + 0.02t^3 + \cos(3\sin t)), \quad t\in [0,2\pi).
\end{equation}
We see that the red  dash-dot line  passes through all the red circles.

\subsection{Recurrent nonuniform samples and   derivatives}

In this part we consider a kind of recurrent  multichannel interpolation which involves nonuniform
samples   and   derivatives. Let $b_1(n) = e^{\i n \alpha}$ and $b_2(n) = \i n$, then we have a matrix defined by
\begin{equation*}
	\mathbf{H}_n= \begin{bmatrix}
		e^{\i n \alpha} & \i n\\
		e^{\i (n+m_0)\alpha} & \i (n+m_0)
	\end{bmatrix}\  \text{for} \ n\in I_1.
\end{equation*}
It is easy to get its inverse as
\begin{equation*}
	\mathbf{H}_n^{-1}= \begin{bmatrix}
		\frac{e^{-\i n \alpha } (m_0+n)}{m_0+n - n e^{\i m_0 \alpha } } & -\frac{ n e^{-\i n \alpha } }{m_0+n-n e^{\i m_0 \alpha } } \\
		\frac{\i e^{\i m_0 \alpha }}{m_0 +n - n e^{\i m_0 \alpha }} & -\frac{\i}{m_0 +n- ne^{\i m_0 \alpha }} \\
	\end{bmatrix}.
\end{equation*}
Set
\begin{align*}
	v_{1,n,\alpha}(t)&:=\frac{e^{-\i n \alpha } (m_0+n)}{m_0+n - n e^{\i m_0 \alpha } } e^{\i n t}-\frac{ n e^{-\i n \alpha } e^{\i (n+m_0)t} }{m_0+n-n e^{\i m_0 \alpha } } =\frac{\left(m_0- n \left(e^{\i m_0 t}-1\right)\right) e^{\i n (t-\alpha )}}{m_0+ n-n e^{\i m_0 \alpha  }},\\
	v_{2,n,\alpha}(t)& :=\frac{\i e^{\i m_0 \alpha }e^{\i n t}}{m_0+n - n e^{\i m_0 \alpha } } -\frac{ \i e^{\i (n+m_0)t} }{m_0+n-n e^{\i m_0 \alpha } } =\frac{\i \left(e^{\i ( m_0\alpha+n t)}-e^{\i  (m_0+n)t}\right)}{m_0+n-n e^{\i  m_0 \alpha }}.
\end{align*}
%\begin{equation*}
%v_{1,n}(t) =\frac{e^{-\i n \alpha } (m_0+n)}{m_0+n - n e^{\i m_0 \alpha } } e^{\i n t}-\frac{ n e^{-\i n \alpha } e^{\i (n+m_0)t} }{m_0+n-n e^{\i m_0 \alpha } } =\frac{\left(m_0- n \left(e^{\i m_0 t}-1\right)\right) e^{\i n (t-\alpha )}}{n+m_0-n e^{\i \alpha  m_0}}
%\end{equation*}
It should be noted that if $\alpha =\frac{\pi}{m_0}$, then the denominator in $v_{i,n,\alpha}(t)$, i.e.,   $m_0+n - n e^{\i m_0 \alpha }=m_0+2n$  would become  $0$ for   $n=-\frac{m_0}{2}$.  Otherwise, we have the interpolating functions
%\begin{align*}
%y_{1,\alpha}(t-t_p)& = \sum_{n=N_1}^{N_1+m_0-1} v_{1,n,\alpha}(t) e^{-\i n \frac{2\pi p}{m_0}},\\
%y_{2,\alpha}(t-t_p)& = \sum_{n=N_1}^{N_1+m_0-1} v_{2,n,\alpha}(t) e^{-\i n \frac{2\pi p}{m_0}}
%\end{align*}
\begin{equation*}
	y_{k,\alpha}(t-t_p) = \frac{1}{m_0} \sum_{n=N_1}^{N_1+m_0-1} v_{k,n,\alpha}(t) e^{-\i n \frac{2\pi p}{m_0}}, \quad k=1,2.
\end{equation*}
Moreover,  a bandlimited function $f\in B_{\mathbf{N}}$  can be exactly reconstructed by
\begin{equation}\label{recurrent_nonuniform}
\mathcal{T}_2(f,2m_0,\alpha,t) := \sum_{p=0}^{m_0-1}f(\alpha + t_p)y_{1,\alpha}(t-t_p) + f'(t_p)y_{2,\alpha}(t-t_p),
\end{equation}
provided that  $m_0\geq \frac{\mu(I^{\mathbf{N}})}{2}$.
\begin{figure*}[!t]
	\centering
	% Requires \usepackage{graphicx}
	\includegraphics[width=4.5in]{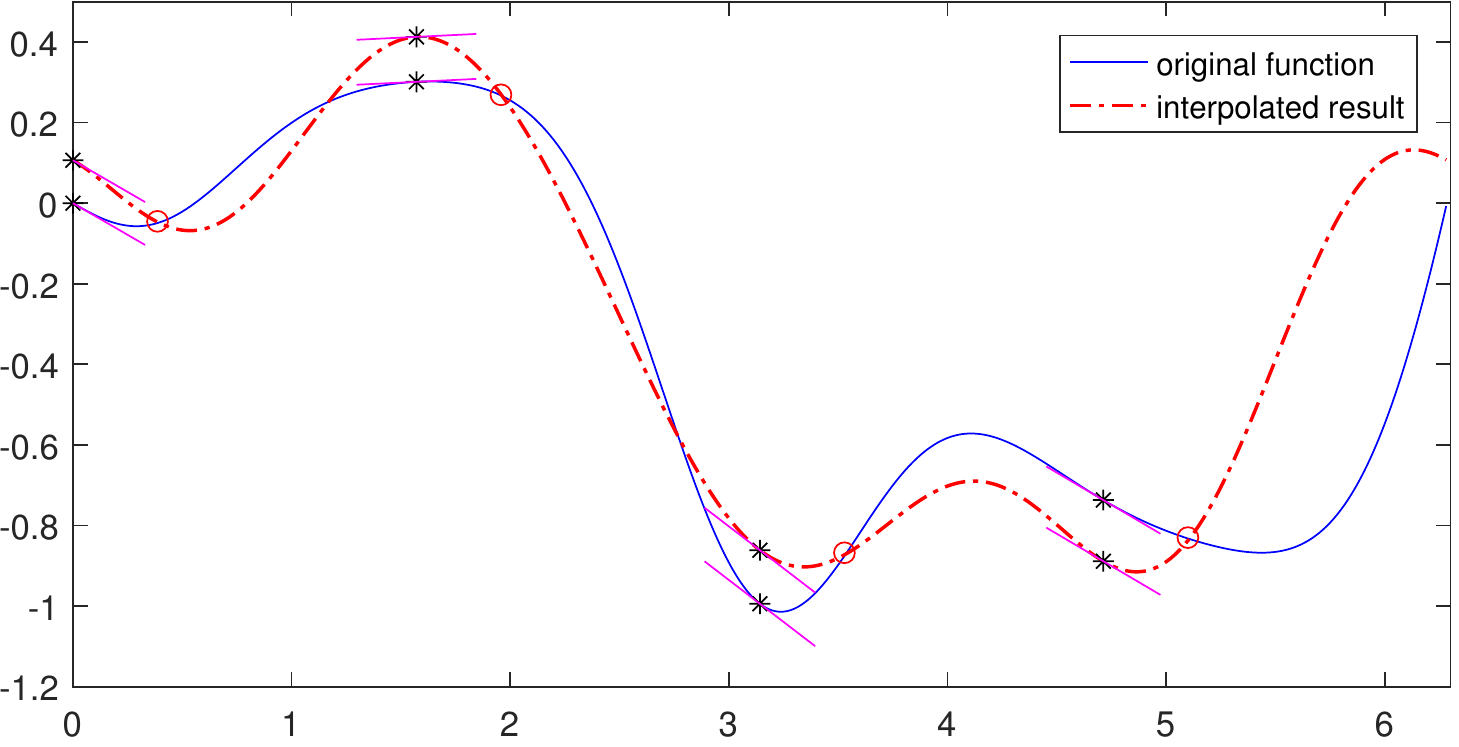}\\
	\caption{ Illustration of interpolation for  recurrent nonuniform samples of a function and its derivative. The blue line is  original function.  The red  dash-dot line  is  interpolated result for the given data points.}\label{figinterlaced_consistency}
\end{figure*}
As mentioned in \cite{chengkou2018multi1}, the interpolating  functions $y_{k,\alpha}(t-t_p)$  for $k=1,2$ can be calculated by taking FFT for $ v_{k,n,\alpha}(t)$  with respect to $n$. When $\alpha =0$, the formula (\ref{recurrent_nonuniform}) reduces to a kind of multichannel interpolation for uniformly distributed data  $\{f(t_p)\}, \{f'(t_p)\}$:
\begin{equation*}%\label{uniformM}
	\mathcal{T}_2(f,2m_0,0,t) \sum_{p=0}^{m_0-1}f( t_p)y_{1,0}(t-t_p) + f'(t_p)y_{2,0}(t-t_p),
\end{equation*}
where
\begin{align*}
	y_{1,0}(t) & = \frac{e^{\i N_1 t}(e^{\i m_0 t}-1)^2(N_1+m_0-(N_1+m_0-1)e^{\i t})}{m_0^2(1-e^{\i t})^2},\\
	y_{2,0}(t)& = \frac{\i e^{\i {N_1} t} \left(2 e^{\i  {m_0}t}-e^{2\i  {m_0} t}-1\right)}{{m_0^2} \left(e^{\i t}-1\right)}.
\end{align*}
We illustrate the   interpolation formula (\ref{recurrent_nonuniform}) in Figure \ref{figinterlaced_consistency} for  recurrent non-uniformly distributed data of $f(t)$ given by  (\ref{ex-signal}). The red circles represent the samples of $f(t)$. The reconstructed function (in red dash-dot line) passes through all the red circles. Besides,  the blue line and   red dash-dot line  have the same slope at the particular positions (shown by  black asterisks), and the $t$-coordinates of red circles and black asterisks are interlaced and bunched.

\section{Multichannel  interpolation of generic non-uniformly distributed data}\label{S4}

Although referred to as nonuniform, there are restrictions on location of samples for recurrent nonuniform sampling case. The distribution of samples, to some extent,  is still regular. Moreover, as mentioned in \cite{sommen2008relationship}, recurrent nonuniform samples can be regarded as a combination of several mutual delayed sequences of uniform samples. In this part, we consider a   more general interpolation formula which is applicable to generic non-uniformly distributed data.
Thanks to the finite summation in (\ref{drictexpression}), it is possible to consider a   specific  case. Let $M=\mu(I^{\mathbf{N}})$, $I_1=\{N_1\}$, $K=1$  in (\ref{drictexpression}), then we construct a matrix
\begin{equation*}
	\mathbf{H}= \begin{bmatrix}
		b_1(N_1)&b_2(N_1)  &\cdots  & b_M(N_1) \\
		b_1(N_1+1)&b_2(N_1+1)  & \cdots&b_M(N_1+1)  \\
		\vdots& \vdots & \ddots & \vdots \\
		b_1(N_1+M-1)& b_2(N_1+M-1) &  \cdots & b_M(N_1+M-1)
	\end{bmatrix}.
\end{equation*}
Under this setting, one may drive various nonuniform sampling interpolation formulas provided that $\mathbf{H}$ is invertible. The key points are how to determine whether $\mathbf{H}$ is invertible and how to calculate the inverse. Unlike $\mathbf{H}_n$ in the normal case, $\mathbf{H}$  is   a large complex-valued matrix  with  high condition number in general. Therefore, in order  to achieve a stable reconstruction,  it is not feasible to compute the inverse of $\mathbf{H}$ by  numerical methods.

\subsection{Generic nonuniform samples}

Let $0\leq t_1< t_2<\dots<t_M<2\pi$   and $b_p(n)=e^{\i n {t_p}}$ for $1\leq p \leq M$. We have the following matrix:
\begin{equation*}
	\mathbf{H}= \begin{bmatrix}
		e^{\i N_1 t_1}& e^{\i N_1 t_2}  &\cdots  & e^{\i N_1 t_M} \\
		e^{\i (N_1+1) t_1}& e^{\i (N_1+1) t_2} & \cdots& e^{\i (N_1+1) t_M}  \\
		\vdots& \vdots & \ddots & \vdots \\
		e^{\i (N_1+M-1) t_1}&e^{\i (N_1+M-1) t_2} &  \cdots &e^{\i (N_1+M-1) t_M}
	\end{bmatrix}.
\end{equation*}
It is easy to show that the determinant of $\mathbf{H}$ is
\begin{equation*}
	\det \mathbf{H} = e^{\i N_1 (t_1+t_2+\cdots+t_M)} \prod_{1 \leq p <q \leq M} \left(e^{\i N_1t_p} - e^{\i N_1t_q} \right) \neq 0.
\end{equation*}
That means that $\mathbf{H}$ is invertible. Denote   by $ z_{p}(k)$ the $(p,k)$-th element of $\mathbf{H}^{-1}$. It can be shown that
\begin{equation*}
	z_{p}(k) = \frac{(-1)^{k+1}e^{-\i N_1  t_p}}{\displaystyle\prod_{\substack{1\leq s\leq M\\ s\neq p}}(e^{\i t_s}-e^{\i t_p})}
	\sum_{\substack{1\leq s_1<s_2<\cdots<s_{M-k}\leq M\\ s_1,s_2,\cdots,s_{M-k}\neq p}} e^{\i (t_{s_1}+t_{s_2}+\cdots+t_{s_{M-k}})}.
\end{equation*}
Therefore  we conclude that a periodic bandlimited function
$f\in B_{\mathbf{N}}$ can be exactly reconstructed from its $M \geq \mu(I^{\mathbf{N}})$ non-uniformly distributed samples. The interpolation formula is given by
\begin{equation}\label{nonuniforminterp}
\mathcal{T}_3(f,M,t):= \sum_{p=1}^{M}f(t_p)h_p(t)
\end{equation}
where
\begin{equation*}
	h_p(t):=\sum_{k=1}^{M}z_{p}(k)e^{\i(k+N_1-1)t}.
\end{equation*}

\begin{figure*}[!t]
	\centering
	% Requires \usepackage{graphicx}
	\includegraphics[width=4.5in]{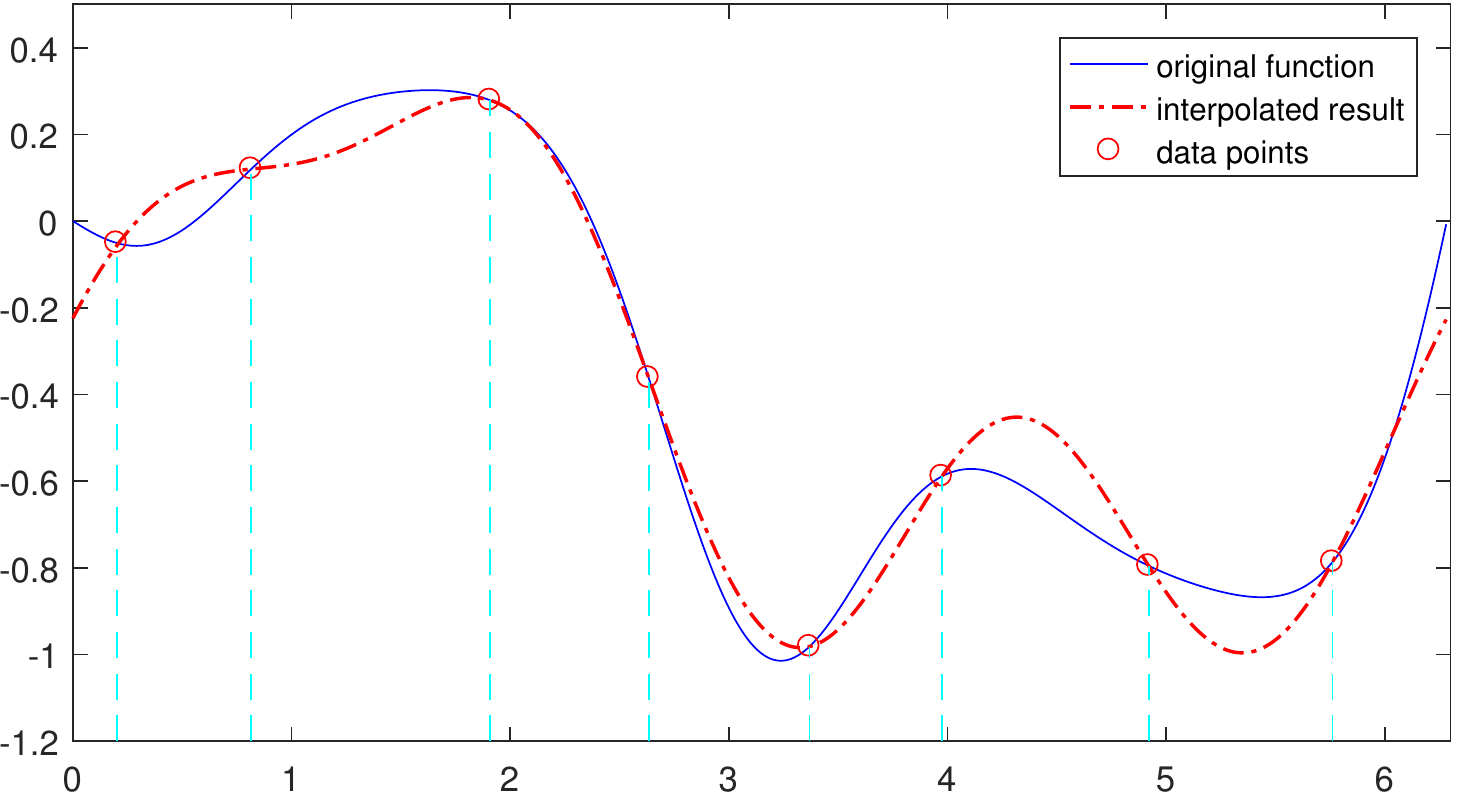}\\
	\caption{ Illustration of   interpolation  for generic nonuniform samples. The blue line is  original function.  The red  dash-dot line  is  interpolated result for the given data points.}\label{fignonuniform11}
\end{figure*}

We can compute $N(>M)$ function values of $h_p(t)$ by taking $N$ fast Fourier transform for $\{z_{p}(k)\}_k$ through zero padding. By applying some trigonometric identities,  the interpolating function $h_p(t)$ can be   simplified into a simper form. By a few basic calculations,
\begin{align*}
	e^{\i t}-e^{\i t_s} =& \cos t-\cos t_s+\i (\sin t-\sin t_s)
	= 2\i \sin(\frac{t-t_s}{2})e^{\i \frac{t+t_s}{2}}.
\end{align*}
Therefore
\begin{equation}\label{tri-ide1}
\prod_{s=1,s\neq p}^{M}(e^{\i t}-e^{\i t_s})=(2\i)^{M-1}e^{\i \frac{M-1}{2}t}
\prod_{s=1,s\neq p}^{M}\sin\left(\frac{t-t_s}{2} \right) e^{\i \frac{t_s}{2}}.
\end{equation}
Note that the left hand side of  (\ref{tri-ide1}) is a trigonometric polynomial with respect to $t$.
Expanding the product, we have
\begin{equation*}
	\prod_{s=1,s\neq p}^{M}(e^{\i t}-e^{\i t_s}) = \sum_{k=1}^M\beta_k e^{\i(k-1)t}
\end{equation*}
where
\begin{equation*}
	\beta_k = (-1)^{M-k}\sum_{\substack{1\leq s_1<s_2<\cdots<s_{M-k}\leq M\\ s_1,s_2,\cdots,s_{M-k}\neq p}}e^{\i (t_{s_1}+t_{s_2}+\cdots+t_{s_{M-k}})}.
\end{equation*}
By similar arguments to (\ref{tri-ide1}), we have
\begin{equation*}\label{tri-ide2}
	\prod_{s=1,s\neq p}^{M}(e^{\i t_s}-e^{\i t_p}) =(-2\i)^{M-1} e^{\i \frac{M-1}{2}t_p}
	\prod_{s=1,s\neq p}^{M}\sin\left(\frac{t_p-t_s}{2} \right) e^{\i \frac{t_s}{2}}.
\end{equation*}
It follows that
\begin{align*}
	h_p(t)=&\sum_{k=1}^{M}z_{p}(k)e^{\i(k+N_1-1)t}
	= \sum_{k=1}^{M} \frac{(-1)^{k+1}e^{-\i N_1  t_p}}{\displaystyle\prod_{\substack{1\leq s\leq M\\ s\neq p}}(e^{\i t_s}-e^{\i t_p})}(-1)^{M-k}\beta_ke^{\i(k+N_1-1)t}  \nonumber\\
	= & \frac{(-1)^{M+1} e^{\i N_1  (t-t_p)}}{\displaystyle\prod_{\substack{1\leq s\leq M\\ s\neq p}}(e^{\i t_s}-e^{\i t_p})}\sum_{k=1}^M\beta_k e^{\i(k-1)t}
	=  \frac{(-1)^{M+1} e^{\i N_1  (t-t_p)}}{\displaystyle\prod_{\substack{1\leq s\leq M\\ s\neq p}}(e^{\i t_s}-e^{\i t_p})}\prod_{s=1,s\neq p}^{M}(e^{\i t}-e^{\i t_s})\nonumber\\
	= & e^{\i N_1  (t-t_p)} e^{\frac{\i (M-1)(t-t_p)}{2}}
	\frac{\prod_{s=1,s\neq p}^{M}\sin\left(\frac{t-t_s}{2} \right)}
	{\prod_{s=1,s\neq p}^{M}\sin\left(\frac{t_p-t_s}{2} \right)}. \nonumber %\label{exhp}
\end{align*}
This formula is consistent with the result presented in \cite{margolis2008nonuniform} by selecting specific values for  parameters $N_1$ and $M$. In comparison to the proof of this result in \cite{margolis2008nonuniform}, the proposed derivation is simpler and more  understandable.

We illustrate the   interpolation formula (\ref{nonuniforminterp}) in Figure \ref{fignonuniform11} for   nonuniform samples of $f(t)$ given by  (\ref{ex-signal}). The red circles represent the randomly selected  nonuniform samples of $f(t)$. The reconstructed function (in red dash-dot line) passes through all the red circles.  For the case $t_p=\frac{2\pi (p-1)}{M}$, $0\leq p \leq M-1$, the formula (\ref{nonuniforminterp}) reduces to the uniform sampling interpolation given in \cite{chengkou2018multi1}.

\subsection{Generic nonuniform samples   and  derivatives}

The fact that a bandlimited function could be reconstructed from the values of the function and its derivative is well known \cite{papoulis1977generalized}. However, the samples involved in such a theorem are uniformly distributed.
Let $t_1, t_2,\dots,t_{m_0}$ be arbitrary $m_0$ non-uniformly spaced points on $[0,2\pi)$. Suppose that
$f\in B_{\mathbf{N}}$ with $\mu(I^{\mathbf{N}})\leq M=2 m_0$.
There is a question of whether $f$ can be perfectly reconstructed from the samples of itself and its first derivative (i.e., $\{f(t_p),f'(t_p)\}_{p=1}^{m_0}$). It is tantamount to asking whether  $\mathbf{H}$ is invertible. Here $\mathbf{H} = [v_{kj}]$ is a $M$-th order square matrix with

\begin{equation*}%\label{H1}
	v_{kj} :=\begin{cases}
		e^{\i (N_1+k-1)t_p}, & \quad j =2p-1;\\
		\i (N_1+k-1)e^{\i (N_1+k-1)t_p}, &\quad j=2p.
	\end{cases}
\end{equation*}
% \begin{equation*}
% \mathbf{H}= \begin{bmatrix}
% e^{\i N_1 t_1}& e^{\i N_1 t_2}  &\cdots  & e^{\i N_1 t_M} \\
% e^{\i (N_1+1) t_1}& e^{\i (N_1+1) t_2} & \cdots& e^{\i (N_1+1) t_M}  \\
% \vdots& \vdots & \ddots & \vdots \\
% e^{\i (N_1+2M-1) t_1}&e^{\i (N_1+2M-1) t_2} &  \cdots &e^{\i (N_1+M-1) t_M}
% \end{bmatrix}
% \end{equation*}
The answer is affirmative. In this subsection, we   derive the main result of  the current paper: interpolation for non-uniformly distributed samples of a function and its derivative. The interpolating functions are presented in closed-form   and the error of reconstructing a non-bandlimited function by proposed formula will be discussed in the next section.

Let $\widetilde{\mathbf{H}} = [\widetilde{v}_{kj}]$ with
\begin{equation*}%\label{H2}
	\widetilde{v}_{kj} :=\begin{cases}
		e^{\i (k-1)t_p}, & \quad j =2p-1;\\
		(N_1+k-1)e^{\i (k-1)t_p}, &\quad j=2p.
	\end{cases}
\end{equation*}
It is easy to see that
\begin{equation*}
	\det\mathbf{H} = (\i)^{m_0}e^{2\i N_1(t_1+t_2+\cdots+t_{m_0}) } \det\widetilde{\mathbf{H}} .
\end{equation*}
%Note that
Note that $\det\widetilde{\mathbf{H}}$ is a function of $t_1, t_2,\dots,t_{m_0}$. The following lemma gives a recursive relation of $\det\widetilde{\mathbf{H}}$.

\begin{lemma}
	Let $\widetilde{\mathbf{H}}$ be given above. Then its determinant satisfies the following recursive relation
	\begin{equation}\label{recurrenceH}
	\det{\widetilde{\mathbf{H}}}(t_1,t_2,\dots,t_{m_0}) = e^{\i t_1} \prod_{p>1}^{m_0}(e^{\i t_p} - e^{\i t_1})^4 \det {\widetilde{\mathbf{H}}}(t_2,\dots,t_{m_0}).
	\end{equation}
\end{lemma}

\begin{proof}
	Applying some column operations to $\widetilde{\mathbf{H}}$, it follows that $\det{\widetilde{\mathbf{H}}}$ is equal  to  the determinant of following matrix % $\det \widetilde{\mathbf{H}}_1$, where $\widetilde{\mathbf{H}}_1$ equals
	\begin{equation}\label{simplfied-matrix}
	\mathbf{C}=\begin{bmatrix}
	1	&0&\cdots  &1	&0 \\
	e^{\i t_1}	&e^{\i t_1}  &\cdots  &e^{\i t_{m_0}}	&e^{\i t_{m_0}} \\
	e^{\i 2t_1}	&2e^{\i 2t_1}  &\cdots  &e^{\i 2t_{m_0}}	&2e^{\i 2t_{m_0}} \\
	\vdots	& \vdots & \ddots & \vdots	&\vdots\\
	e^{\i (2m_0-1)t_1}	&(2m_0-1)e^{\i (2m_0-1)t_1}  &\cdots &  e^{\i (2m_0-1)t_{m_0}}	&(2m_0-1)e^{\i (2m_0-1)t_{m_0}}
	\end{bmatrix}.
	\end{equation}
	Subtracting the multiple $e^{\i t_1}$ of row $(k-1)$ from row $k$  for $k=2m_0,2m_0-1,\cdots,2$ successively, we remove  first column    without  changing the determinant:
	\begin{multline*}%\label{midquati1}
		\det\left[
		\begin{array}{cccc}
			e^{\i t_1} &x_2 & 0 \cdot x_2  e^{\i t_2} + e^{\i t_2} &\cdots\\
			e^{\i 2t_1}&x_2 e^{\i t_2}  &1\cdot   x_2e^{\i t_2}  + e^{\i 2t_2}&\cdots  \\
			\vdots & \vdots&\vdots&\ddots \\
			e^{\i (2m_0-1)t_1} &x_2e^{\i (2m_0-2)t_2} &(2m_0-2)x_2e^{\i t_2}  + e^{\i (2m_0-2)t_2}&\cdots\\
		\end{array}\right.\\
		\left.
		\begin{array}{cc}
			x_{m_0} & 0 \cdot x_{m_0}  e^{\i t_{m_0}} + e^{\i t_{m_0}}\\
			x_{m_0} e^{\i t_{m_0}}  &1\cdot   x_{m_0}e^{\i t_{m_0}}  + e^{\i 2t_{m_0}}  \\
			\vdots&\vdots \\
			x_{m_0}e^{\i (2m_0-2)t_{m_0}} &(2m_0-2)x_{m_0}e^{\i t_{m_0}}  + e^{\i (2m_0-2)t_{m_0}}\\
		\end{array}
		\right]
	\end{multline*}
	where $x_k=e^{\i t_k}-e^{\i t_1}$ for $k=2,3,\cdots,m_0$.
	Subtracting  the multiple $\frac{e^{\i t_k}}{x_k}$ of column $(2k-2)$ from column $(2k-1)$    and extracting $x_k$ from column $(2k-2)$ and $(2k-1)$ for $k=2,3,\cdots,m_0$ successively, we reach the  result of
	\begin{equation}\label{Hrecur1}
	\det\widetilde{\mathbf{H}}= x_2^2x_3^2\cdots x_{m_0}^2\det\widetilde{\mathbf{H}}^{(1)}
	\end{equation}
	where $\widetilde{\mathbf{H}}^{(1)}$ equals
	\begin{equation*}
		\begin{bmatrix}
			e^{\i t_1}& 1 & 0  & \cdots& 1 & 0 \\
			e^{\i 2t_1}& e^{\i t_2} & e^{\i t_2}  & \cdots & e^{\i t_{m_0}} & e^{\i t_{m_0}} \\
			e^{\i 3t_1}& e^{\i 2t_2} & 2 e^{\i2 t_2}  &\cdots  & e^{\i 2t_{m_0}} & 2 e^{\i2 t_{m_0}} \\
			\vdots& \vdots & \vdots & \ddots& \vdots & \vdots\\
			e^{\i (2m_0-1)t_1}& e^{\i(2m_0-2) t_2} & (2m_0-2) e^{\i(2m_0-2) t_{2}} &\cdots & e^{\i(2m_0-2) t_{m_0}} & (2m_0-2) e^{\i(2m_0-2) t_{m_0}}
		\end{bmatrix}
	\end{equation*}
	For $\widetilde{\mathbf{H}}^{(1)}$, extracting $e^{\i t_1}$ from the first column and subtracting the multiple $e^{\i t_1}$ of row $(k-1)$ from row $k$  for $k=2m_0-1,2m_0-2,\cdots,2$ successively, we remove  first column of $\widetilde{\mathbf{H}}^{(1)}$   and reach the result of
	\begin{equation}\label{Hrecur2}
	\det \widetilde{\mathbf{H}}^{(1)} =e^{\i t_1} \det \widetilde{\mathbf{H}}^{(2)}
	\end{equation}
	where $\widetilde{\mathbf{H}}^{(2)}$ equals
	\begin{multline*}%\label{midquati2}
		\left[
		\begin{array}{ccc}
			x2& e^{\i t_2}  & \cdots  \\
			x_2 e^{\i t_2} & x_2 e^{\i t_2}  +e^{\i 2 t_2}&\cdots\\
			\vdots & \vdots & \ddots\\
			x_2 e^{\i(2m_0-3) t_2} & (2m_0-3) x_2e^{\i t_{2}}+e^{\i(2m_0-3) t_{2}} &\cdots\\
		\end{array}\right.\\
		\left.
		\begin{array}{cc}
			x_{m_0} & e^{\i t_{m_0}} \\
			x_{m_0} e^{\i t_{m_0}} & x_{m_0} e^{\i t_{m_0}}+ e^{\i 2 t_{m_0}}\\
			\vdots&\vdots \\
			x_{m_0} e^{\i(2m_0-3) t_{m_0}} & (2m_0-3)x_{m_0}  e^{\i t_{m_0}}+e^{\i(2m_0-3) t_{m_0}}\\
		\end{array}
		\right].
	\end{multline*}
	%\begin{equation*}
	%\begin{bmatrix}
	% 1 & 0  & \cdots& 1 & 0 \\
	% x2& e^{\i t_2}  & \cdots &  x_{m_0} & e^{\i t_{m_0}} \\
	% x_2 e^{\i t_2} & x_2 e^{\i t_2}  +e^{\i 2 t_2}&\cdots  & x_{m_0} e^{\i t_{m_0}} & x_{m_0} e^{\i t_{m_0}}+ e^{\i 2 t_{m_0}}\\
	% \vdots & \vdots & \ddots& \vdots & \vdots\\
	%x_2 e^{\i(2m_0-3) t_2} & (2m_0-3) x_2e^{\i t_{2}}+e^{\i(2m_0-3) t_{2}} &\cdots & x_{m_0} e^{\i(2m_0-3) t_{m_0}} & (2m_0-3)x_{m_0}  e^{\i t_{m_0}}+e^{\i(2m_0-3) t_{m_0}}
	%\end{bmatrix}
	%\end{equation*}
	Subtracting  the multiple $\frac{e^{\i t_k}}{x_k}$ of column $(2k-3)$ from column $(2k-2)$    and extracting $x_k$ from column $(2k-3)$ and $(2k-2)$ for $k=2,3,\cdots,m_0$ successively, we  get
	\begin{equation}\label{Hrecur3}
	\det \widetilde{\mathbf{H}}^{(2)} = x_2^2x_3^2\cdots x_{m_0}^2\det{\widetilde{\mathbf{H}}}(t_2,\dots,t_{m_0}).
	\end{equation}
	Then the recursive relation (\ref{recurrenceH}) follows from (\ref{Hrecur1}), (\ref{Hrecur2}) and (\ref{Hrecur3}). The proof is complete.
\end{proof}

Since $\det{\widetilde{\mathbf{H}}}(t_{m_0}) = \det
\begin{bmatrix}
1 & N_1 \\
e^{\i t_{m_0}} & (N_1+1)e^{\i t_{m_0}}
\end{bmatrix} = e^{\i t_{m_0}} $. By induction, we conclude that
\begin{equation*}
	\det{\widetilde{\mathbf{H}}}(t_1,t_2,\dots,t_{m_0}) = e^{\i (t_1+t_2+\cdots+t_{m_0})}\prod_{1 \leq p <q \leq m_0}(e^{\i t_q} - e^{\i t_p})^4 .
\end{equation*}
It follows that
\begin{equation*}
	\det\mathbf{H} = (\i)^{m_0} e^{\i(2N_1+1)(t_1+t_2+\cdots+t_{m_0})}\prod_{1 \leq p <q \leq m_0}(e^{\i t_q} - e^{\i t_p})^4 \neq 0.
\end{equation*}
Therefore $\mathbf{H}$ is invertible.
Let $w_{j}(k)$ denote the $(j,k)$ element of $\mathbf{H}^{-1}$.
We define the interpolating functions $\phi_{p}(t)$ and $\psi_{p}(t)$ as follows:
\begin{align}
	\phi_{p}(t)= & \sum_{k=1}^{2m_0} w_{2p}(k) e^{\i (N_1+k-1)t},\label{interpolate-f-phi} \\
	\psi_{p}(t)=&  \sum_{k=1}^{2m_0} w_{2p-1}(k)e^{\i (N_1+k-1)t}.\label{interpolate-f-psi}
\end{align}
Then  we have a theorem  about nonuniform multichannel interpolation as follows.

\begin{theorem}
	Let $0\leq t_1< t_2<\dots<t_{m_0}<2\pi$  be non-uniformly distributed points.	Suppose that
	$f\in B_{\mathbf{N}}$ with $\mu(I^{\mathbf{N}})\leq M=2 m_0$. Then it can be exactly recovered by the following interpolation formula
	\begin{equation}\label{nonuniform-multichannel}
	\mathcal{T}_4(f,2m_0,t) =\sum_{p=1}^{m_0}f(t_p)\psi_{p}(t)+f'(t_p)\phi_{p}(t).
	\end{equation}
\end{theorem}

%Computation of $\widetilde{\mathbf{H}}^{-1}$.
To derive the closed form expressions of $\psi_{p}$ and $\phi_{p}$, the direct approach is to compute the inverse of $\mathbf{H}$. This  is, as discussed earlier,  not a feasible approach. Fortunately, the interpolating functions can be computed tactfully by introducing some auxiliary matrices.
Firstly, we need to compute cofactor matrix of $\mathbf{C}$ defined by (\ref{simplfied-matrix}).
Constructing  an auxiliary matrix $\mathbf{A}$ by substituting the second column of $\mathbf{C}$ with $[1,e^{\i t},\cdots,e^{\i (2m_0-1)t}]$ will bring convenience to the computation:
\begin{equation*}%\label{auxiliary-matrix}
	\mathbf{A}=\begin{bmatrix}
		1	&1&\cdots  &1	&0 \\
		e^{\i t_1}	&e^{\i t}  &\cdots  &e^{\i t_{m_0}}	&e^{\i t_{m_0}} \\
		e^{\i 2t_1}	&e^{\i 2t}  &\cdots  &e^{\i 2t_{m_0}}	&2e^{\i 2t_{m_0}} \\
		\vdots	& \vdots & \ddots & \vdots	&\vdots\\
		e^{\i (2m_0-1)t_1}	&e^{\i (2m_0-1)t}  &\cdots &  e^{\i (2m_0-1)t_{m_0}}	&(2m_0-1)e^{\i (2m_0-1)t_{m_0}}
	\end{bmatrix}.
\end{equation*}
On the one hand, by similar arguments to the computation of $\det \widetilde{\mathbf{H}}$, we get that
%\begin{equation}
% \det \mathbf{A}(t,t_1,t_2,\cdots,t_{m_0}) =(e^{\i t}-e^{\i t_1})\left( \prod_{s>1}^{m_0}(e^{\i t}-e^{\i t_s})^2 \right) \left( \prod_{q>1}^{m_0}(e^{\i t_q}-e^{\i t_1})^2 \right)
%\end{equation}
\begin{align}\label{first-exression}
	&\det \mathbf{A}(t,t_1,t_2,\cdots,t_{m_0}) \nonumber \\
	=&(e^{\i t}-e^{\i t_1})\left( \prod_{s>1}^{m_0}(e^{\i t}-e^{\i t_s})^2 \right) \left( \prod_{q>1}^{m_0}(e^{\i t_q}-e^{\i t_1})^2 \right)\det{\widetilde{\mathbf{H}}}(t_2,\dots,t_{m_0}).
\end{align}
On the other hand,  the cofactor expansion of $\det\mathbf{A}$ along the second column gives:
\begin{equation}\label{second-exression}
\det \mathbf{A}(t,t_1,t_2,\cdots,t_{m_0})
= \sum_{k=1}^{2m_0}C_{k2}(t_1,t_2,\cdots,t_{m_0}) e^{\i (k-1) t}
\end{equation}
where $C_{k2}$ is the $(k,2)$ cofactor of $\mathbf{C}$. By comparing the coefficients of $e^{\i (k-1)t}$  in (\ref{first-exression}) and (\ref{second-exression}), we obtain the expression of $C_{k2}$ for $k=1,2,\cdots,2m_0$. For example,
\begin{align*}%\label{C_{12}}
	&C_{12}(t_1,t_2,\cdots,t_{m_0}) \nonumber\\
	=&-e^{\i (t_1+2t_2+2t_3+\cdots+2t_{m_0})}\left( \prod_{q>1}^{m_0}(e^{\i t_q}-e^{\i t_1})^2 \right)\det{\widetilde{\mathbf{H}}}(t_2,\dots,t_{m_0}) \nonumber \\
	=&-e^{\i (t_1+3t_2+3t_3+\cdots+3t_{m_0})} \left( \prod_{q>1}^{m_0}(e^{\i t_q}-e^{\i t_1})^2 \right)\prod_{2 \leq p <q \leq m_0}(e^{\i t_q} - e^{\i t_p})^4.
\end{align*}
%\begin{equation*}
%C_{12}(t_1,t_2,\cdots,t_{m_0}) = -e^{\i (t_1+2t_2+2t_3+\cdots+2t_{m_0})}\left( \prod_{q>1}^{m_0}(e^{\i t_q}-e^{\i t_1})^2 \right)\det{\widetilde{\mathbf{H}}}(t_2,\dots,t_{m_0})
%\end{equation*}
Let $H_{kj}$ denote the $(k,j)$ cofactor of $\mathbf{H}$. Note that $\mathbf{C}$ can be constructed from $\mathbf{H}$ by applying some column operations. We immediately have the following relations:
\begin{align}
	H_{k,2p}(t_1,t_2,\cdots,t_{m_0})
	=&(\i)^{m_0-1}\left[ \prod_{s=1}^{m_0}e^{\i N_1 t_s}\right] \left[ \prod_{r=1,r\neq p}^{m_0}e^{\i N_1 t_r}\right]  C_{k,2p}(t_1,t_2,\cdot,t_{m_0})\label{H2p} \\
	H_{k,2p-1}(t_1,t_2,\cdots,t_{m_0})=&- \frac{\partial H_{k,2p}}{\partial t_p}(t_1,t_2,\cdots,t_{m_0}).\label{H1p}
\end{align}
%Let $w_{jk}$ denote the $(j,k)$ element of $\mathbf{H}^{-1}$.
It is well known that the elements of $\mathbf{H}^{-1}$ can be expressed by  cofactors of $\mathbf{H}$, namely
\begin{equation}\label{Fouriercoeff}
w_{j}(k) = \frac{H_{kj}}{\det \mathbf{H}}.
\end{equation}
%We define the interpolation functions $\phi_{p}(t)$ and $\psi_{p}(t)$ as follows:
%\begin{align}
%\phi_{p}(t)= & \sum_{k=1}^{2m_0} w_{2p,k} e^{\i (N_1+k-1)t}\label{interpolate-f-phi} \\
%\psi_{p}(t)=&  \sum_{k=1}^{2m_0} w_{2p-1,k} e^{\i (N_1+k-1)t}\label{interpolate-f-psi}
%\end{align}
%Then  we have the following non-uniform multichannel interpolation formula:
%\begin{equation}\label{nonuniform-multichannel}
%\hat{f}(t)=\mathcal{T}_Mf(t) =\sum_{p=1}^{m_0}f(t_p)\psi_{p}(t)+f'(t_p)\phi_{p}(t).
%\end{equation}
%If $f\in B_{\mathbf{N}}$ with $\mu(I^{\mathbf{N}})\leq M=2 m_0$, then (\ref{nonuniform-multichannel}) exactly
%reconstruct $f(t)$ without error, namely
%\begin{equation*}
%f(t) = \hat{f}(t).
%\end{equation*}
By similar arguments to (\ref{first-exression}) and (\ref{second-exression}), we have that
\begin{align}
	&\sum_{k=1}^{2m_0}C_{k,2p}(t_1,t_2,\cdots,t_{m_0})e^{\i (k-1)t} \nonumber\\
	=&(e^{\i t}-e^{\i t_p})\left[ \prod_{s=1,s\neq p}^{m_0}e^{\i t_s}(e^{\i t}-e^{\i t_s})^2  (e^{\i t_s}-e^{\i t_p})^2  \right] \prod_{\substack{1\leq s_1<s_2< m_0\\ s_1,s_2\neq p}}(e^{\i t_{s_1}} - e^{\i t_{s_2}})^4. \label{sum_of_C}
\end{align}
%\begin{equation*}
%(e^{\i t}-e^{\i t_p})\left[ \prod_{s=1,s\neq p}^{m_0}e^{\i t_s}(e^{\i t}-e^{\i t_s})^2  (e^{\i t_s}-e^{\i t_p})^2  \right] \prod_{\substack{1\leq s_1<s_2< m_0\\ s_1,s_2\neq p}}(e^{\i t_{s_1}} - e^{\i t_{s_2}})^4.
%\end{equation*}
%\begin{multline}
%\sum_{k=1}^{2m_0}C_{k,2p}(t_1,t_2,\cdots,t_{m_0})e^{\i (k-1)t}=(e^{\i t}-e^{\i t_p})\left[ \prod_{s=1,s\neq p}^{m_0}(e^{\i t}-e^{\i t_s})^2\right]\left[ \prod_{q=1,q\neq p}^{m_0}(e^{\i t_q}-e^{\i t_p})^2\right]  \\
%\times  \left[ \prod_{r=1,r\neq p}^{m_0}e^{\i t_r}\right] \prod_{\substack{1\leq s_1<s_2< m_0\\ s_1,s_2\neq p}}(e^{\i t_{s_1}} - e^{\i t_{s_2}})^4
%\end{multline}
%
%\begin{align}\label{summationC}
%&\sum_{k=1}^{2m_0}C_{k,2p}(t_1,t_2,\cdots,t_{m_0})e^{\i (k-1)t} \nonumber\\
%=&(e^{\i t}-e^{\i t_p})\left[ \prod_{s=1,s\neq p}^{m_0}(e^{\i t}-e^{\i t_s})^2\right]
%\left[ \prod_{q=1,q\neq p}^{m_0}(e^{\i t_q}-e^{\i t_p})^2\right]   \left[ \prod_{r=1,r\neq p}^{m_0}e^{\i t_r}\right] \prod_{\substack{1\leq s_1<s_2< m_0\\ s_1,s_2\neq p}}(e^{\i t_{s_1}} - e^{\i t_{s_2}})^4 \nonumber
%\end{align}
Plugging (\ref{H2p}) and (\ref{Fouriercoeff})  into  (\ref{interpolate-f-phi}) and applying (\ref{sum_of_C}),
%By the definition of $\phi_{p}(t)$ given by (\ref{interpolate-f-phi}) and applying (\ref{H2p}),
it follows that
\begin{equation*}
	\phi_{p}(t) = -\i e^{\i N_1(t-t_p)}(e^{\i(t-t_p)}-1)
	\left[\prod_{s=1,s\neq p}^{m_0} (e^{\i t}-e^{\i t_s})^{2} (e^{\i t_s}-e^{\i t_p})^{-2} \right].
\end{equation*}

More efforts are needed to compute $\psi_{p}(t) $ due to the partial derivative operation in (\ref{H1p}). For simplicity, we denote Eq.(\ref{sum_of_C})    and $\frac{\partial}{\partial t_p} \prod_{q=1,q\neq p}^{m_0} (e^{\i t_q} - e^{\i t_p})$ by $\xi_p(t)$  and  $\gamma_p$ respectively. By some direct computations, we have   that
$$\gamma_p = -\i \sum_{s=1,s\neq p}^{m_0} \
\prod_{\substack{1\leq q\leq m_0\\ q\neq p,q\neq s}}(e^{\i t_q}-e^{\i t_p})$$
and
\begin{align}
	\frac{\partial \xi_p(t)}{\partial t_p} =&\sum_{k=1}^{2m_0}\frac{\partial}{\partial t_p}C_{k,2p}(t_1,t_2,\cdots,t_{m_0})e^{\i (k-1)t} \nonumber\\
	=&2\gamma_p(e^{\i t}-e^{\i t_p})\left[ \prod_{s=1,s\neq p}^{m_0}e^{\i t_s}(e^{\i t}-e^{\i t_s})^2(e^{\i t_s}-e^{\i t_p})\right]\prod_{\substack{1\leq s_1<s_2< m_0\\ s_1,s_2\neq p}}(e^{\i t_{s_1}} - e^{\i t_{s_2}})^4 \nonumber\\
	& ~~~~~~~~~~~~~~~~~~~~~~~~~~~~~~~~~~ -\i \xi_p(t) e^{\i t_p} (e^{\i t}-e^{\i t_p})^{-1} .
	\label{diff_xi}
\end{align}
From (\ref{H2p}) and (\ref{H1p}), it follows that
\begin{multline}
	H_{k,2p-1}(t_1,t_2,\cdots,t_{m_0}) =-\i^{m_0-1}  e^{\i N_1(t_1+t_2+\cdots t_{m_0})} \left[\prod_{r=1,r\neq p}^{m_0}e^{\i N_1 t_r}\right]
	\frac{\partial}{\partial t_p}C_{k,2p}(t_1,t_2,\cdots,t_{m_0}) \\
	-\i^{m_0} N_1 e^{\i N_1 t_p}
	\left[\prod_{s=1,s\neq p}^{m_0}e^{2\i N_1 t_s}\right]C_{k,2p}(t_1,t_2,\cdots,t_{m_0}).
	\label{newH1p}
\end{multline}
Plugging (\ref{newH1p}) and (\ref{Fouriercoeff}) into (\ref{interpolate-f-psi}), and applying (\ref{diff_xi}) and (\ref{sum_of_C}), we get that
%By similar arguments to the computation of $\phi_{p}(t)$, we have that
\begin{multline}\label{exprpsi}
	\psi_{p}(t) = 2e^{\i N_1(t-t_p)}(e^{\i t}-e^{\i t_p})\left[\prod_{s=1,s\neq p}^{m_0}(e^{\i t}-e^{\i t_s})^2(e^{\i t_s}-e^{\i t_p})^{-3} \right] \sum_{s=1,s\neq p}^{m_0}
	\prod_{\substack{1\leq q\leq m_0\\ q\neq p,q\neq s}}(e^{\i t_q}-e^{\i t_p})\\
	-\i N_1 \phi_{p}(t) +e^{\i N_1(t-t_p)} \prod_{s=1,s\neq p}^{m_0}(e^{\i t}-e^{\i t_s})^2(e^{\i t_s}-e^{\i t_p})^{-2}.
\end{multline}
%\begin{equation*}
%\psi_{p}(t) = -\i N_1 \phi_{p}(t) +e^{\i N_1(t-t_p)}
%\end{equation*}
\begin{figure*}[!t]
	\centering
	% Requires \usepackage{graphicx}
	\includegraphics[width=4.5in]{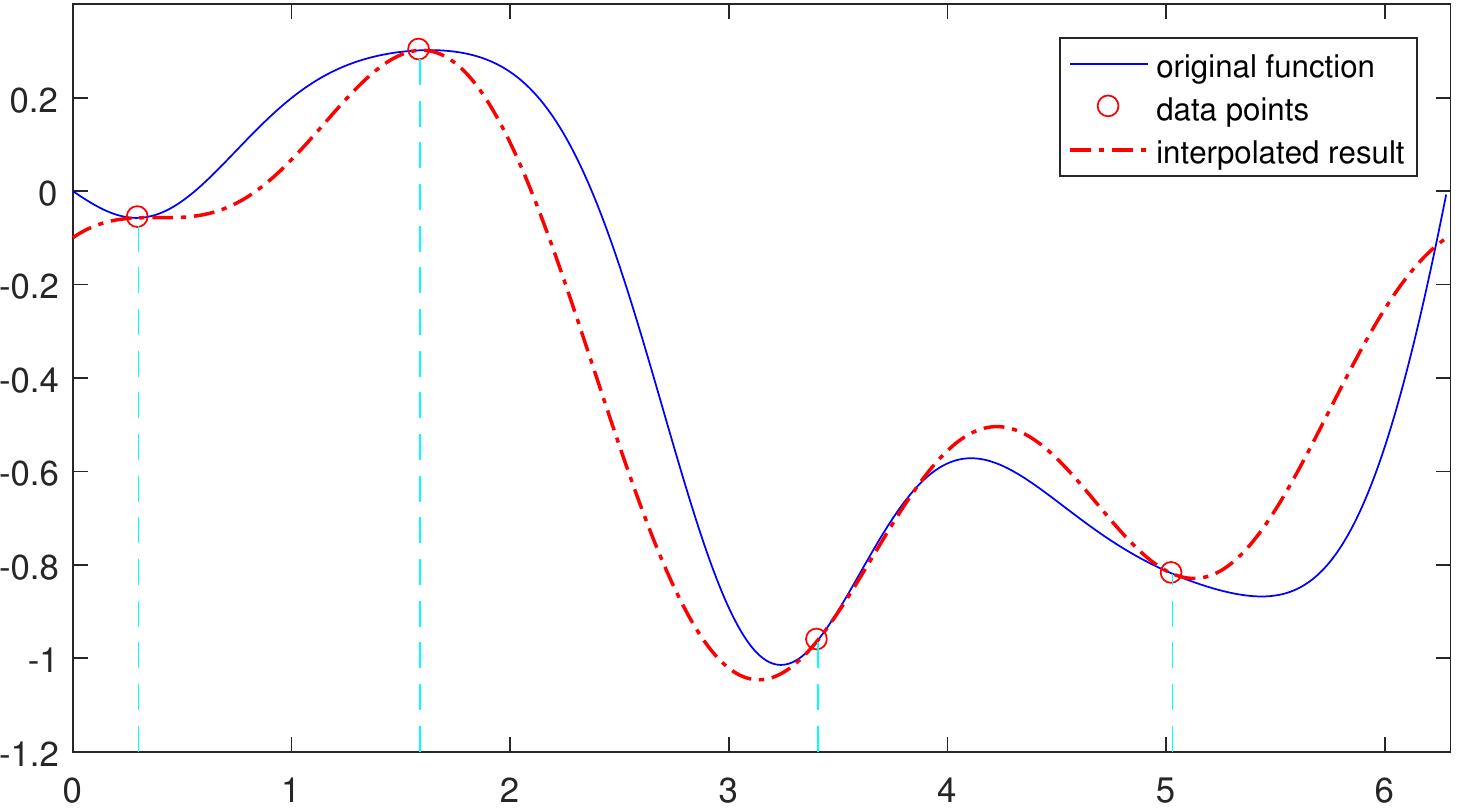}\\
	\caption{ Illustration of  interpolation for generic nonuniform samples of a function and its derivative. The blue line is  original function. The red circle is  given data points.  The red  dash-dot line  is    interpolated result for the given data points.}\label{figconsistency_verifying}
\end{figure*}

Next we shall verify that $\phi_{p}$ and $\psi_{p}$ satisfy the following interpolation consistency:
%\begin{equation*}
%\begin{cases}
\begin{align}
	\phi_{p}(t_p)&=0, ~\phi_{p}(t_q)=0,~ \phi_{p}'(t_p)=1,~ \phi_{p}'(t_q)=0, \label{consist1}\\
	\psi_p(t_p)&=1, ~\psi_p(t_q)=0, ~\psi_p'(t_p)=0,~ \psi_p'(t_q)=0, \label{consist2}
\end{align}
%\end{cases}
%\end{equation*}
for  $1\leq p\neq q \leq m_0$. This consistency guarantees that
\begin{equation*}
	f(t_p) =\mathcal{T}_4(f,2m_0,t_p),\quad  f'(t_p) = \left.\frac{\partial \mathcal{T}_4(f,2m_0,t)}{\partial t}\right|_{t=t_p} ,\quad p=1,2,\dots,m_0,
\end{equation*}
even if    the reconstruction is not exact. We only give the validation of (\ref{consist1}) and omit the proof of (\ref{consist2}) for the sake of brevity.   It is obvious that $\phi_{p}(t_p) = \phi_{p}(t_q)=0$. Let
$$x_p(t)=-\i e^{\i N_1(t-t_p)}(e^{\i(t-t_p)}-1), ~~~z_p(t) =\prod_{s=1,s\neq p}^{m_0} (e^{\i t}-e^{\i t_s})^{2} (e^{\i t_s}-e^{\i t_p})^{-2}. $$
It is easy to check that $x_p'(t_p)=z_p(t_p) =1$ and $x_p(t_p)=z_p(t_q)=z_p'(t_q)=0$. Therefore
\begin{align*}
	\phi_p'(t_p) &= x_p'(t_p)z_p(t_p)+x_p(t_p)z_p'(t_p)=1,\\
	\phi_p'(t_q) &= x_p'(t_q)z_p(t_q)+x_p(t_q)z_p'(t_q)=0.
\end{align*}
Figure \ref{figconsistency_verifying} illustrates multichannel interpolation of non-uniformly distributed data and its  interpolation consistency. The blue line displays
function given by (\ref{ex-signal}).
%\begin{equation}\label{ex-signal}
%  f(t)=0.05t(t-2\pi)(0.04t^2 + 0.02t^3 + \cos(3\sin t)).
%\end{equation}
The nonuniform grid points are randomly selected as $(t_1,t_2,t_3,t_4) = (0.2998,1.5866,3.4062,5.0281)$. The red dash-dot line presents
the interpolated result for $(t_p,f(t_p))$, $p =1,2,3,4$. We can see that
not only  the red dash-dot line pass through all the data points but also  it is tangent  to the blue line at each point.

\section{Numerical examples and error analysis}\label{S5}

\subsection{Numerical examples}

According to the  types of samples, we abbreviate the interpolation formulas (\ref{recurrent_nonuniform1}),  (\ref{recurrent_nonuniform}),  (\ref{nonuniforminterp}) and (\ref{nonuniform-multichannel}) as  RN1, RN2,  GN1 and GN2 respectively for simplicity. Specially, (\ref{nonuniforminterp}) and (\ref{nonuniform-multichannel}) are respectively  abbreviated as U1 and
U2 if the samples are uniformly spaced. Figure \ref{inclusion} illustrates the inclusion relations of these formulas.
\begin{figure*}[!ht]
	\centering
	% Requires \usepackage{graphicx}
	\includegraphics[width=3.6in]{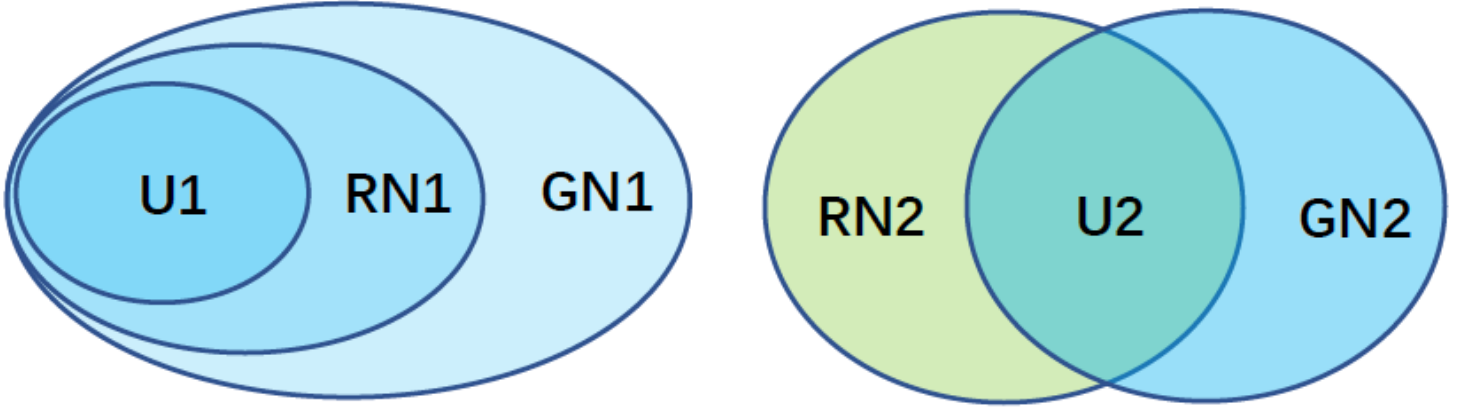}\\
	\caption{ Illustration of inclusion relations  for the interpolation formulas.}\label{inclusion}
\end{figure*}

We use  the aforementioned formulas to reconstruct non-bandlimited functions. As in \cite{chengkou2018multi1}, we select
\begin{equation*}
	\Phi(z) = \frac{0.08 z^2+0.06z^{10}}{(1.3-z)(1.5-z)} + \frac{0.05z^3+0.09z^{10}}{(1.2+z)(1.3+z)}
\end{equation*}
as the test function.  Let $f(t)=\Re[\Phi(e^{\i t}) ]$, then its Hilbert transform is
$\mathcal{H}f(t) = \Im[\Phi(e^{\i t}) ]$ by the theory of Hardy space. In the following, we compare the performance of  proposed several formulas for reconstructing   $f$ and $\mathcal{H}f$. The results are listed in Table \ref{approximatedtable}. Denote by $\hat{f}(t)$ the reconstructed result, then the relative mean square error (RMSE) is given by
\begin{equation*}
	\begin{split}
		\delta_1 = &\left. \left( \int_{\T}\abs{f(t)-\hat{f}(t)}^2dt\right)^{\frac{1}{2}} \middle /\left( \int_{\T}\abs{f(t)}^2dt\right)^{\frac{1}{2}} \right.  \\
		\approx &
		\left. \left( \sum_{p=0}^{2047}\abs{f(\tfrac{2\pi p}{2048})-\hat{f}(\tfrac{2\pi p}{2048})}^2\right)^{\frac{1}{2}} \middle / \left(\sum_{p=0}^{2047}\abs{f(\tfrac{2\pi p}{2048})}^2\right)^{\frac{1}{2}} \right. .
	\end{split}
\end{equation*}
\begin{table}[!t]
	\caption{Reconstruction results   using the different interpolation formulas}\label{approximatedtable}
	\centering
	\vspace{0.1cm}
	\begin{tabular}{c  c  c c  l l }
		\hline
		{Total samples} &  $f$ & $f'$ & {Pattern} &~~~~ $\delta_1$~(Variance)&~~~$\delta_2$ ~(Variance)  \\ \hline
		$36$ & $36$ & $0$ & RN1  & $0.8560$ & $0.8358$\\
		$36$ & $36$ & $0$ & GN1 & $0.5548$ ($0.0026$) & $0.5535$ ($0.0026$)\\
		$36$ & $36$ & $0$ & U1  & $0.5120$ & $0.5116$ \\
		$36$ & $18$ & $18$ & RN2  & $0.6163$ & $0.6159$ \\
		$36$ & $18$ & $18$ & GN2 & $1.0752$ ($0.0346$) & $1.0550$ ($0.0361$)\\
		$36$ & $18$ & $18$ & U2  & $0.9241$ & $0.8381$ \\
		$54$ & $54$ & $0$ & RN1  & $0.1955$ & $0.1922$\\
		$54$ & $54$ & $0$ & GN1 & $0.1501$ ($1.11\times 10^{-4}$) & $0.1498$ ($1.09\times 10^{-4}$) \\
		$54$ & $54$ & $0$ & U1  & $0.1376$ & $0.1376$ \\
		$54$ & $27$ & $27$ & RN2  & $0.1830$ & $0.1830$ \\
		$54$ & $27$ & $27$ & GN2 & $0.2829$ ($3.4\times 10^{-3}$) & $0.2794$ ($3.4\times10^{-3}$)\\
		$54$ & $27$ & $27$ & U2  & $0.2582$ & $0.2483$ \\
		$72$ & $72$ & $0$ & RN1  & $0.0437$ & $0.0426$\\
		$72$ & $72$ & $0$ & GN1 & $0.0321$ ($3.77\times 10^{-6}$) & $0.0320$ ($3.73\times 10^{-6}$) \\
		$72$ & $72$ & $0$ & U1  & $0.0291$ & $0.0291$ \\
		$72$ & $36$ & $36$ & RN2  & $0.0355$ & $0.0355$ \\
		$72$ & $36$ & $36$ & GN2 & $0.0647$ ($1.18\times 10^{-4}$) & $0.0636$ ($1.22\times 10^{-4}$)\\
		$72$ & $36$ & $36$ & U2  & $0.0557$ & $0.0520$ \\
		$108$ & $108$ & $0$ & RN1  & $0.0018$ & $0.0017$\\
		$108$ & $108$ & $0$ & GN1 & $0.0013$ ($4.56\times 10^{-9}$) & $0.0013$ ($4.50\times 10^{-9}$) \\
		$108$ & $108$ & $0$ & U1  & $0.0012$ & $0.0012$ \\
		$108$ & $54$ & $54$ & RN2  & $0.0014$ & $0.0014$ \\
		$108$ & $54$ & $54$ & GN2 & $0.0028$ ($1.81\times 10^{-7}$) & $0.0028$ ($1.83\times 10^{-7}$)\\
		$108$ & $54$ & $54$ & U2  & $0.0023$ & $0.0021$ \\
		\hline
	\end{tabular}
\end{table}
Similarly, we denote by $\delta_2$  the RMSE for reconstructing the Hilbert transform $\mathcal{H}f$. In the experiments, $\alpha$ is selected as $\frac{\pi}{N}$ for RN1 and RN2, if the total number of samples is $N$.  For GN1 and GN2, the nonuniform grids are randomly generated by
\begin{align}
	t_n &=(n-1)\frac{2\pi}{N}+\zeta_n,\quad n=1,2,\dots,N \label{rand1}\\
	\tilde{t}_n &=(n-1)\frac{4\pi}{N}+\eta_n,\quad n=1,2,\dots,\frac{N}{2}\label{rand2}
\end{align}
where $\zeta_n$ and $\eta_n$ are i.i.d. sequences of random variables with uniform distribution on
$(0,\frac{2\pi}{3N})$ and $(0,\frac{4\pi}{3N})$ respectively. To give a more comprehensive presentation for GN1 and GN2, we repeat each experiment of generic nonuniform sampling for 100 times. Accordingly, $\delta_1$ and $\delta_2$ of GN1 and GN2 are averaged    over these 100 times experiments,  and
the corresponding variances are also provided.

\begin{figure*}[!t]
	\centering
	% Requires \usepackage{graphicx}
	\includegraphics[width=6in]{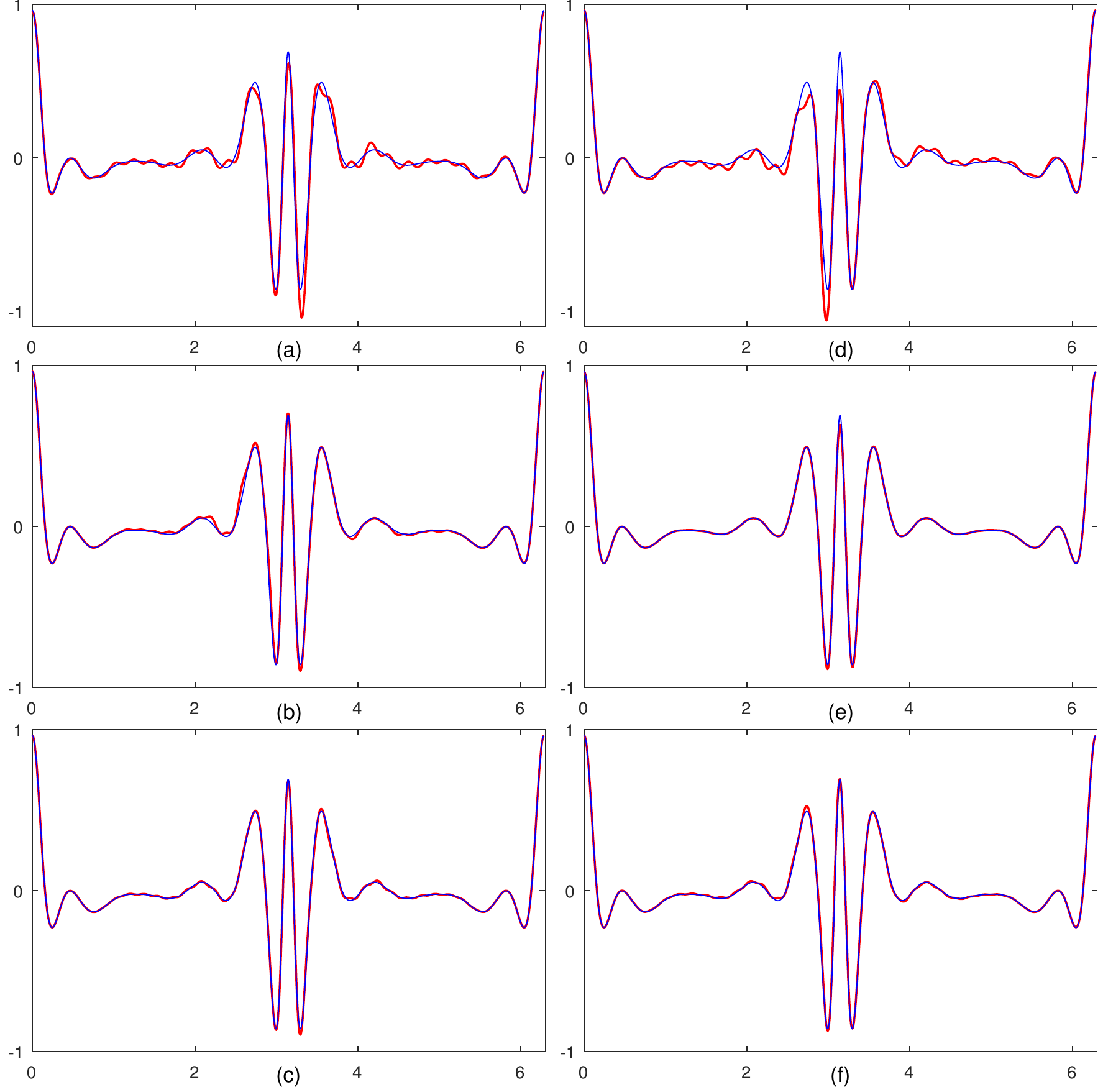}\\
	\caption{Reconstructing $f$ by (a) GN1 with total 54 samples, (b) GN2 with total 72 samples, (c) GN1 with total 72 samples, (d) RN1 with total 54 samples, (e) RN2 with total 72 samples, (f) RN1 with total 72 samples. }\label{fignonuniform}
\end{figure*}

\begin{figure*}[!t]
	\centering
	% Requires \usepackage{graphicx}
	\includegraphics[width=6in]{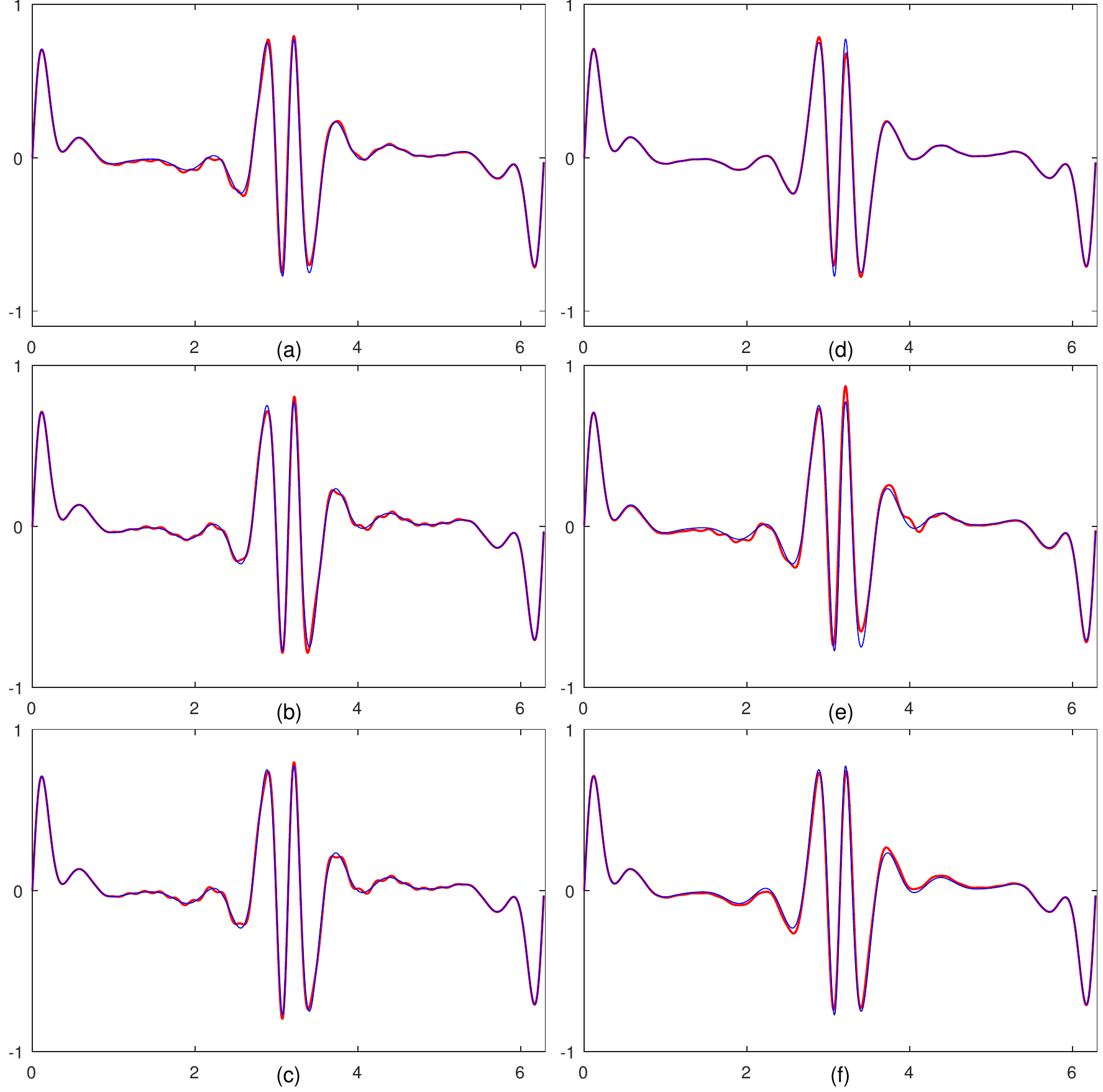}\\
	\caption{Reconstructing $\mathcal{H}f$ by (a) RN1 with total 64 samples, (b) GN1 with total 64 samples, (c) U1 with total 64 samples, (d) RN2 with total 64 samples, (e) GN2 with total 64 samples, (f) U2 with total 64 samples.}\label{fignonuniformH}
\end{figure*}

Some  results for reconstructing $f$ and $\mathcal{H}f$ are depicted in Figure \ref{fignonuniform} and \ref{fignonuniformH}.  Visually there is no much difference among these reconstructed results by different formulas provided that the same number of samples are used. Roughly, some   conclusions could be drawn  from the numerical results as follows.
\begin{enumerate}
	\item   If the same amount of data is employed to reconstruct $f$ (or $\mathcal{H}f$), the fluctuations  of RMSE caused by the different  data types   and  data distribution patterns  are not significant. In other words,  the amount of data is the  chief factor that affects   performance of   the reconstruction.
	\item  The more grid points   the   data is distributed on,  the better  performance of the    reconstruction behave.  We can see this by comparing the reconstructed results of RN2 and GN2.
	In addition,  the more even the  data distribution is, the better performance of the reconstruction behave. We can see this by comparing the reconstructed results of GN1 and U1, or GN2 and U2.
	\item In general, reconstructing a function from its own samples performs slightly better than the reconstruction that involves other types of data. This can be seen from the reconstructed results of GN1 and GN2.
\end{enumerate}
The last two  conclusions are  certainly  based on the premise that the same amount of data is used for reconstruction. And an additional observation is that $\delta_1$ and $\delta_2$ are nearly equal in each experiment， since the Fourier coefficients of $f$ and $\mathcal{H}f$ have the same absolute value for all $n\in \mathbb{Z} \setminus\{0\}$.

%distribution pattern of data

%==============================================================
\subsection{Error analysis}

In the previous subsection, we presented the reconstruction errors for the proposed interpolation formulas experimentally. In this part, we will give the error estimations analytically which are very important to  the reliability of the reconstruction methods.

We denote by $f_{\tau}(t)=f(t-\tau)$  the shifted function of $f$. Let $\mathcal{T}_\mathbf{N}$ be a reconstruction operator corresponding to any one of the  aforementioned interpolation formulas. Here $\mathbf{N}$  represents  the location of Fourier coefficients for reconstructed function  $\mathcal{T}_\mathbf{N}f$.
It is easy to see that
\begin{align}
	\mathcal{T}_\mathbf{N}f_{\tau}(t)& =\sum_{p=1}^{m_0}f(t_p-\tau)\psi_p(t)+f'(t_p-\tau)\phi_p(t) \label{shiftedoperator}\\
	\mathcal{T}_\mathbf{N}f(t-\tau)& =\sum_{p=1}^{m_0}f(t_p)\psi_p(t-\tau)+f'(t_p)\phi_p(t-\tau).\nonumber
\end{align}
Thus $\mathcal{T}_\mathbf{N} f(t-\tau) \neq \mathcal{T}_\mathbf{N} f_{\tau}(t)$. Not   just for GN2,  most of the other interpolation formulas are not shift-invariant   in general.
Therefore the MSE defined by
\begin{equation*}
	\varsigma(f,\mathbf{N},\tau) = \norm{f_{\tau}-\mathcal{T}_\mathbf{N}  f_{\tau}}_2^2= \frac{1}{2\pi}\int_{\T} \abs{f_{\tau}(t)-\mathcal{T}_\mathbf{N}  f_{\tau}(t)}^2 dt
\end{equation*}
is not independent on $\tau$.  There is no doubt that $\varsigma(f,\mathbf{N},\tau)$ is ${2\pi } $ periodic in $\tau$.
Note that the time shift $\tau$ could be viewed as the phase difference of $f$ and $f_{\tau}$. And the exact phase of a function or a signal  is generally unknown  in  most practical applications \cite{jacob2002sampling}. Hence, we need to compute the averaged error
\begin{equation*}
	\varepsilon(f,\mathbf{N}) = \sqrt{\frac{1}{{2\pi }} \int_{0}^{{2\pi}}\varsigma(f,\mathbf{N},\tau)d \tau}.
\end{equation*}

As can be seen from the previous section that the derivation of GN2 is more arduous than the others. In the following, we derive the expression of averaged error for GN2.
% For  $w_{j}(k)$ ($1\leq j,k\leq 2m_0$)  given in (\ref{Fouriercoeff}), we extend it by setting $w_{j}(k)=0$ for $1\leq j \leq 2m_0$ and
%$k\notin \{k:1\leq k \leq 2m_0\}$. To each extended sequence $w_{j}(k)$, one can always associate a periodic sequence defined by
%\begin{equation*}
%\tilde{w}_j(k):= \sum_{s=-\infty}^{\infty} w_{j}(k+1-N_1-2m_0s).
%\end{equation*}
From (\ref{shiftedoperator}), (\ref{interpolate-f-phi}) and (\ref{interpolate-f-psi}), we rewrite
$\mathcal{T}_\mathbf{N}f_{\tau}(t)$ as
\begin{equation*}
	\sum_{k=1}^{2m_0} e^{\i (N_1+k-1)t} \sum_{p=1}^{m_0}\left[f(t_p-\tau)w_{2p-1}(k)+f'(t_p-\tau)w_{2p}(k)\right].
\end{equation*}
It is noted that $\mathbf{N}$ is  equal to $\{N_1,N_1+2m_0-1\}$ in the above formula. Applying  the Parseval's identity, we have that
\begin{equation}\label{key1}
\begin{split}
& \frac{1}{2\pi}\int_{\T}\overline{f_{\tau}(t)}\mathcal{T}_\mathbf{N}f_{\tau}(t)dt\\
=&\sum_{n\in I^\mathbf{N}}\overline{a(n)} e^{\i n \tau}\sum_{p=1}^{m_0}\left[f(t_p-\tau)w_{2p-1}(n-N_1+1)+f'(t_p-\tau)w_{2p}(n-N_1+1)\right].
\end{split}
\end{equation}
%  \begin{align*}
% & \frac{1}{2\pi}\int_{\T}\overline{f_{\tau}(t)}\mathcal{T}_\mathbf{N}f_{\tau}(t)dt\\
% =&\sum_{n\in I^\mathbf{N}}\overline{a(n)} e^{\i n %\tau}\sum_{p=1}^{m_0}\left[f(t_p-\tau)w_{2p-1}(n-N_1+1)+f'(t_p-\tau)w_{2p}(n-N_1+1)\right].
% \end{align*}
Similarly,
\begin{align}
	& \norm{f_{\tau}}_2^2=\sum_{n\in \mathbb{Z}}\abs{a(n)}^2 \label{key2}\\
	&  \norm{\mathcal{T}_\mathbf{N}f_{\tau}}_2^2=\sum_{k=1}^{2m_0} \sum_{p=1}^{m_0}\sum_{q=1}^{m_0}\sum_{j=1}^{4}D_j(p,q,\tau)E_j(p,q,k) \label{key3}
\end{align}
% \begin{equation*}
% \norm{\mathcal{T}_\mathbf{N}f_{\tau}}_2^2=\sum_{k=1}^{2m_0} %\sum_{p=1}^{m_0}\sum_{q=1}^{m_0}\sum_{j=1}^{4}D_j(p,q,\tau)E_j(p,q,k)
% \end{equation*}
where
\begin{align*}
	D_1(p,q,\tau)&=f(t_p-\tau)\overline{f(t_q-\tau)},\quad E_1(p,q,k) = w_{2p-1}(k)\overline{w_{2q-1}(k)}; \\
	D_2(p,q,\tau)&=f(t_p-\tau)\overline{f'(t_q-\tau)},\quad E_2(p,q,k) = w_{2p-1}(k)\overline{w_{2q}(k)}; \\
	D_3(p,q,\tau)&=f'(t_p-\tau)\overline{f(t_q-\tau)},\quad E_3(p,q,k) = w_{2p}(k)\overline{w_{2q-1}(k)}; \\
	D_4(p,q,\tau)&=f'(t_p-\tau)\overline{f'(t_q-\tau)},\quad E_4(p,q,k) = w_{2p}(k)\overline{w_{2q}(k)}.
\end{align*}
To simplify (\ref{key1}), (\ref{key2}) and (\ref{key3}), we need to introduce some  identities:
\begin{align*}
	&\frac{1}{2\pi}\int_{\T}f(t_p-\tau)e^{\i n \tau}d\tau =a(n)e^{\i n t_p},\\
	&\frac{1}{2\pi}\int_{\T}f'(t_p-\tau)e^{\i n \tau}d\tau =\i n a(n)e^{\i n t_p},\\
	& \frac{1}{2\pi}\int_{\T}f(t_p-\tau)\overline{f(t_q-\tau)}d\tau=\sum_{n\in\mathbb{Z}}\abs{a(n)}^2e^{\i n(t_p-t_q)},\\
	&\frac{1}{2\pi}\int_{\T}f(t_p-\tau)\overline{f'(t_q-\tau)}d\tau=-\i\sum_{n\in\mathbb{Z}}\abs{a(n)}^2ne^{\i n(t_p-t_q)},\\
	&\frac{1}{2\pi}\int_{\T}f'(t_p-\tau)\overline{f'(t_q-\tau)}d\tau=\sum_{n\in\mathbb{Z}}\abs{a(n)}^2n^2e^{\i n(t_p-t_q)}.
\end{align*}
Integrating   the both sides of (\ref{key1}) and (\ref{key3}) on $\T$ with respect to $\tau$ and making use of the above identities, we get that
\begin{align*}
	&	\frac{1}{4\pi^2}\int_{\T}d\tau \int_{\T}\overline{f_{\tau}(t)}\mathcal{T}_\mathbf{N}f_{\tau}(t)dt\\
	=& \sum_{n\in I^\mathbf{N}}\abs{a(n)}^2 \sum_{p=1}^{m_0}\left( e^{\i n t_p}w_{2p-1}(n-N_1+1)+\i n e^{\i n t_p}w_{2p}(n-N_1+1)\right)
\end{align*}
and
\begin{align*}
	&	\frac{1}{2\pi}\int_{\T}\norm{\mathcal{T}_\mathbf{N}f_{\tau}}_2^2d\tau \\
	=&\sum_{k=1}^{2m_0} \sum_{n\in \mathbb{Z}}\abs{a(n)}^2 \sum_{1\leq p,q \leq m_0}\Big(e^{\i n (t_p-t_q)}  E_1(p,q,k)-\i n e^{\i n (t_p-t_q)}  E_2(p,q,k)\\
	& ~~~~~~~~~~~~~~~~~~~~~~~~~~~~~~~ +\i n e^{\i n (t_p-t_q)}  E_3(p,q,k)+ n^2e^{\i n (t_p-t_q)}  E_3(p,q,k)\Big)\\
	=&\sum_{k=1}^{2m_0} \sum_{n\in \mathbb{Z}}\abs{a(n)}^2 \abs{\sum_{p=1}^{m_0} \left(e^{\i nt_p}w_{2p-1}(k)+\i n e^{\i nt_p}w_{2p}(k)\right) }^2.
\end{align*}
From the definition of $w_j(k)$, for any $n\in I^\mathbf{N}$
\begin{align*}
	& \sum_{p=1}^{m_0}\left( e^{\i n t_p}w_{2p-1}(n-N_1+1)+\i n e^{\i n t_p}w_{2p}(n-N_1+1)\right)=1,\\
	& \sum_{p=1}^{m_0} \left(e^{\i nt_p}w_{2p-1}(k)+\i n e^{\i nt_p}w_{2p}(k)\right)=1.
\end{align*}
It follows that
\begin{align*}
	\frac{1}{4\pi^2}\int_{\T}d\tau \int_{\T}{f_{\tau}(t)}\overline{\mathcal{T}_\mathbf{N}f_{\tau}(t)}dt=	\frac{1}{4\pi^2}\int_{\T}d\tau \int_{\T}\overline{f_{\tau}(t)}\mathcal{T}_\mathbf{N}f_{\tau}(t)dt
	= \sum_{n\in I^\mathbf{N}}\abs{a(n)}^2
\end{align*}
and
\begin{align*}
	&\frac{1}{2\pi}\int_{\T}\norm{\mathcal{T}_\mathbf{N}f_{\tau}}_2^2d\tau\\
	= &
	\sum_{n\in I^\mathbf{N}}\abs{a(n)}^2 +\sum_{n\notin I^\mathbf{N}}\abs{a(n)}^2\sum_{k=1}^{2m_0}\abs{\sum_{p=1}^{m_0} \left(e^{\i nt_p}w_{2p-1}(k)+\i n e^{\i nt_p}w_{2p}(k)\right) }^2.
\end{align*}
\begin{figure*}[!t]
	\centering
	% Requires \usepackage{graphicx}
	\includegraphics[width=4.5in]{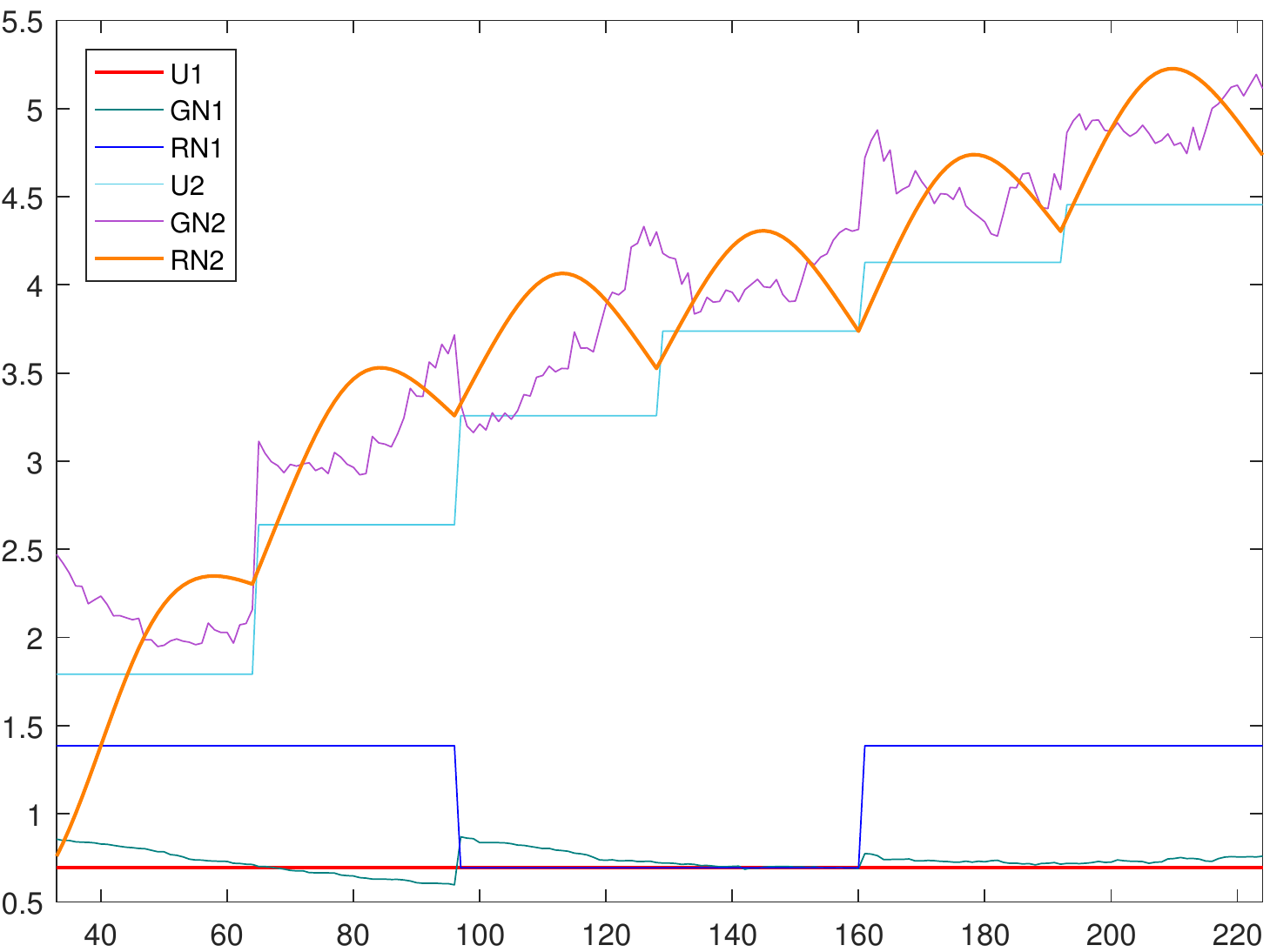}\\
	\caption{ Illustration of $\log Er(\mathbf{N},n)$ for U1, GN1, RN1, U2, GN2, RN2 respectively.}\label{figER64}
\end{figure*}
Therefore the square of the averaged error for GN2 is given by
\begin{align*}
	\varepsilon^2(\text{GN2},f,\mathbf{N})& = \frac{1}{2\pi}\int_{\T}\norm{f_{\tau}}_2^2d\tau - \frac{1}{4\pi^2}\int_{\T}d\tau \int_{\T}{f_{\tau}(t)}\overline{\mathcal{T}_\mathbf{N}f_{\tau}(t)}dt \\
	&\quad -\frac{1}{4\pi^2}\int_{\T}d\tau \int_{\T}\overline{f_{\tau}(t)}\mathcal{T}_\mathbf{N}f_{\tau}(t)dt +\frac{1}{2\pi}\int_{\T}\norm{\mathcal{T}_\mathbf{N}f_{\tau}}_2^2d\tau\\
	& =\sum_{n\notin I^\mathbf{N}} \abs{a(n)}^2 Er(\text{GN2},\mathbf{N},n)
\end{align*}
where
\begin{equation*}
	Er(\text{GN2},\mathbf{N},n) =1+\sum_{k=1}^{2m_0}\abs{\sum_{p=1}^{m_0} \left(e^{\i nt_p}w_{2p-1}(k)+\i n e^{\i nt_p}w_{2p}(k)\right) }^2.
\end{equation*}

\begin{table}[!t]
	\caption{Comparison  of several existing   interpolation methods.}\label{comparesampling}
	\centering
	\vspace{0.1cm}
	\begin{threeparttable}
		\begin{tabular}{p{3.6cm}|p{2.0cm}<{\centering}|p{2.1cm}<{\centering}|p{2.1cm}<{\centering}|p{2.4cm}<{\centering}}	
			\hline
			Different methods	& Untruncated implementation & Applicable to  nonuniform samples&Applicable to  multichannel samples& Closed form of interpolating functions   \\
			\hline
			Proposed method  & yes & yes & yes  & yes  \\
			\hline
			Single-channel interpolation  by FFT  \cite{selva2015fft} & yes &  yes\tnote{1} &  N/A  &  N/A \\
			\hline
			GSE   \cite{papoulis1977generalized,chengkou2018generalized} & N/A &  yes\tnote{2} & yes & yes \\
			\hline
			Classical	nonuniform sampling on real line   \cite{zayed1993advances,seip1987an} &  N/A &  yes&  N/A & N/A \\  \hline
			Single-channel	nonuniform trigonometric interpolation  \cite{margolis2008nonuniform} & yes &  yes&  N/A & yes \\
			\hline
		\end{tabular}
		\begin{tablenotes}
			\footnotesize
			\item[1] The nonuniform samples in \cite{selva2015fft}   have to be located in a regular grid .
			\item[2]  The distribution of nonuniform samples in GSE is  recurrent.
		\end{tablenotes}
	\end{threeparttable}
\end{table}

Similarly, we can get the averaged errors for GN1 and RN2 respectively  as
\begin{align*}
	\varepsilon^2(\text{GN1},f,\mathbf{N})&  =\sum_{n\notin I^\mathbf{N}} \abs{a(n)}^2
	Er(\text{GN1},\mathbf{N},n)\\
	\varepsilon^2(\text{RN2},f,\mathbf{N})& = \sum_{n\notin I^\mathbf{N}} \abs{a(n)}^2
	Er(\text{RN2},\mathbf{N},n)
\end{align*}
where
\begin{align*}
	Er(\text{GN1},\mathbf{N},n) =&1+\sum_{k=1}^{M}\abs{\sum_{p=1}^{M}  e^{\i nt_p} z_p(k)}^2\\
	Er(\text{RN2},\mathbf{N},n) =&1+ \abs{\frac{(2m_0+n-k_nm_0)e^{\i(k_n-1)m_0\alpha}-ne^{\i m_0 \alpha}}{2m_0+n-k_nm_0-(n+m_0-k_nm_0)e^{\i m_0 \alpha}}}^2\\
	&~~+\abs{\frac{n-(m_0+n-k_nm_0)e^{\i(k_n-1)m_0\alpha}}{2m_0+n-k_nm_0-(n+m_0-k_nm_0)e^{\i m_0 \alpha}}}^2
\end{align*}
%\begin{equation*}
%Er(\text{RN2},\mathbf{N},n) =1+ \abs{\frac{(2m_0+n-k_nm_0)e^{\i(k_n-1)m_0\alpha}-ne^{\i m_0 \alpha}}{2m_0+n-k_nm_0-(n+m_0-k_nm_0)e^{\i m_0 \alpha}}}^2+\abs{\frac{(2m_0+n-k_nm_0)e^{\i(k_n-1)m_0\alpha}-ne^{\i m_0 \alpha}}{2m_0+n-k_nm_0-(n+m_0-k_nm_0)e^{\i m_0 \alpha}}}^2.
%\end{equation*}
%\begin{equation*}
% \varepsilon^2(GN2,f,\mathbf{N}) =
%\end{equation*}
with $k_n =\text{fix}\left( \frac{n-N_1}{m_0}\right) +1$ and $\text{fix}(x)$ rounds $x$  to the nearest integer toward zero. Note that the other interpolation formulas can be subsumed in the above three cases,
therefore we obtain all the averaged errors of six aforementioned formulas. The sequences  $Er(\text{U1},\mathbf{N},n)$, $Er(\text{GN1},\mathbf{N},n)$, ..., $Er(\text{RN2},\mathbf{N},n)$ are depicted graphically in Figure \ref{figER64}. Here $\mathbf{N}=(N_1,N_2)=(-31,32)$, thereby $m_0=32, M=64$.  For GN1 and GN2, the nonuniform grids are randomly generated by (\ref{rand1}) and (\ref{rand2}) respectively. The domain for each $Er(\mathbf{N},n)$ plotted  in Figure \ref{figER64} is set as $\{N_2+1\leq n\leq N_2+3\mu(I^{\mathbf{N}})\}$. The theoretical analysis of error is in accord with  the result of numerical examples and therefore the conclusions made in the previous subsection are underpinned.

The proposed interpolation method involves non-uniformly spaced multichannel samples. There are notable existing interpolation methods involving nonuniform or multichannel samples.  We provide the Table \ref{comparesampling}   to compare these existing results.
Among   the numerous sampling or interpolation methods , we  only present several typical types in Table \ref{comparesampling}. It is noted that  the representative references listed here are far from complete.

%\begin{remark}
%		The nonuniform samples in single-channel interpolation by FFT  \cite{selva2015fft} have to be located in a regular grid and the distribution of nonuniform samples in GSE is  recurrent. This implies that these two interpolation methods are partially applicable to nonuniform samples.
%	%The blanks in Table \ref{comparesampling} mean that the characters are not available up to the present.
%	% and  the blanks in  Table \ref{comparesampling} simply represent that the corresponding characters have not been found  so far.
%\end{remark}

\section{Application to image recovery}\label{S6}
In the previous sections, we dealt with techniques for reconstructing a continuous function from different types of discrete samples. In this section, we introduce a simple application of the proposed interpolation formulas to image recovery. To begin, consider Figure \ref{FIGresimage10} (b), which is
severely degraded because of the damaged pixels. Suppose that the damaged pixels are non-uniformly located. The goal of this part is to recover the missing pixels via interpolation.

Note that  the proposed  formulas are  one-dimensional, we have to compute interpolation result for each row of image first, and then apply interpolation for each column by using the same operations. As the distantly separated  image regions are irrelevant virtually, we should treat the reconstruction problem locally. In the following, the test image is set to be Lena ($256\times 256$), and  it is degraded by wiping   out $43.5\%$ randomly selected pixels, see Figure \ref{FIGresimage10} (b).  Each row of image is divided into $32$ equal parts,  namely $8$ pixels per part. Repeating interpolation process through  the image pieces produced by dividing, we obtain   values for all the missing pixels.    Applying the same operations to each column,   we have another reconstructed result. It is noted that the dividing treatment has an additional benefit that  it makes computation complexity linear in the size of image.

It is natural to average two reconstructed results. Besides, we need to convert interpolation result into unsigned 8-bit integer type. A direct way for such a conversion is based on
\begin{equation*}
	{Z}(\mathcal{I}_{xy}) =\begin{cases}
		255&  \text{if} ~~ \mathcal{I}_{xy} \geq 255  \\
		0 & \text{if} ~~ \mathcal{I}_{xy} \leq 0\\
		\text{round}(\mathcal{I}_{xy}) & \text{if} ~~ 0<\mathcal{I}_{xy} <255
	\end{cases}
\end{equation*}
where $\mathcal{I}_{xy}$ is the intensity value at location $(x,y)$.
For a more elaborate conversion, we introduce a correction for the  values produced by interpolation.
Let $\Lambda_{xy}$ be the $3\times 3$ neighborhood centered on $(x,y)$, the correction is defined as
\begin{equation*}
	\hat{\mathcal{I}}_{xy} =\begin{cases}
		\max\left\lbrace \mathcal{I}_{x'y'}:(x',y')\in \Lambda_{xy}\backslash \{(x,y)\}\right\rbrace &  \text{if} ~~ \mathcal{I}_{xy} =\max\left\lbrace  \mathcal{I}_{x'y'}:(x',y')\in \Lambda_{xy}\right\rbrace   \\
		\min\left\lbrace \mathcal{I}_{x'y'}:(x',y')\in \Lambda_{xy}\backslash \{(x,y)\}\right\rbrace &  \text{if} ~~ \mathcal{I}_{xy} =\min\left\lbrace  \mathcal{I}_{x'y'}:(x',y')\in \Lambda_{xy}\right\rbrace   \\
		\mathcal{I}_{xy} & \text{otherwise}
	\end{cases}
\end{equation*}
where $\mathcal{I}_{xy}$ and $\hat{\mathcal{I}}_{xy}$ are the intensity values at location $(x,y)$  before and after correction  respectively. From the definition, this correction  is certain to be convergent after finite iterations. In practice, more fortunately, it can be convergent generally by $3$ or $4$ iterations.

Note that the damaged pixels can be also  viewed as impulse noise (also called salt-and-pepper noise) in an image. It is  known that the median filter, which is a very useful order-statistic filter in image processing, is particularly effective in the reduction of impulse noise \cite{petrou2010image,Gonzalez2006digitalim}.  Basically, to perform median filtering at $(x,y)$ is to determine the median for  values of the pixel in  $\Lambda_{xy}$ and assign that median to $(x,y)$ in the filtered image.

\begin{figure*}[!t]
	\centering
	% Requires \usepackage{graphicx}
	\includegraphics[width=5.6 in]{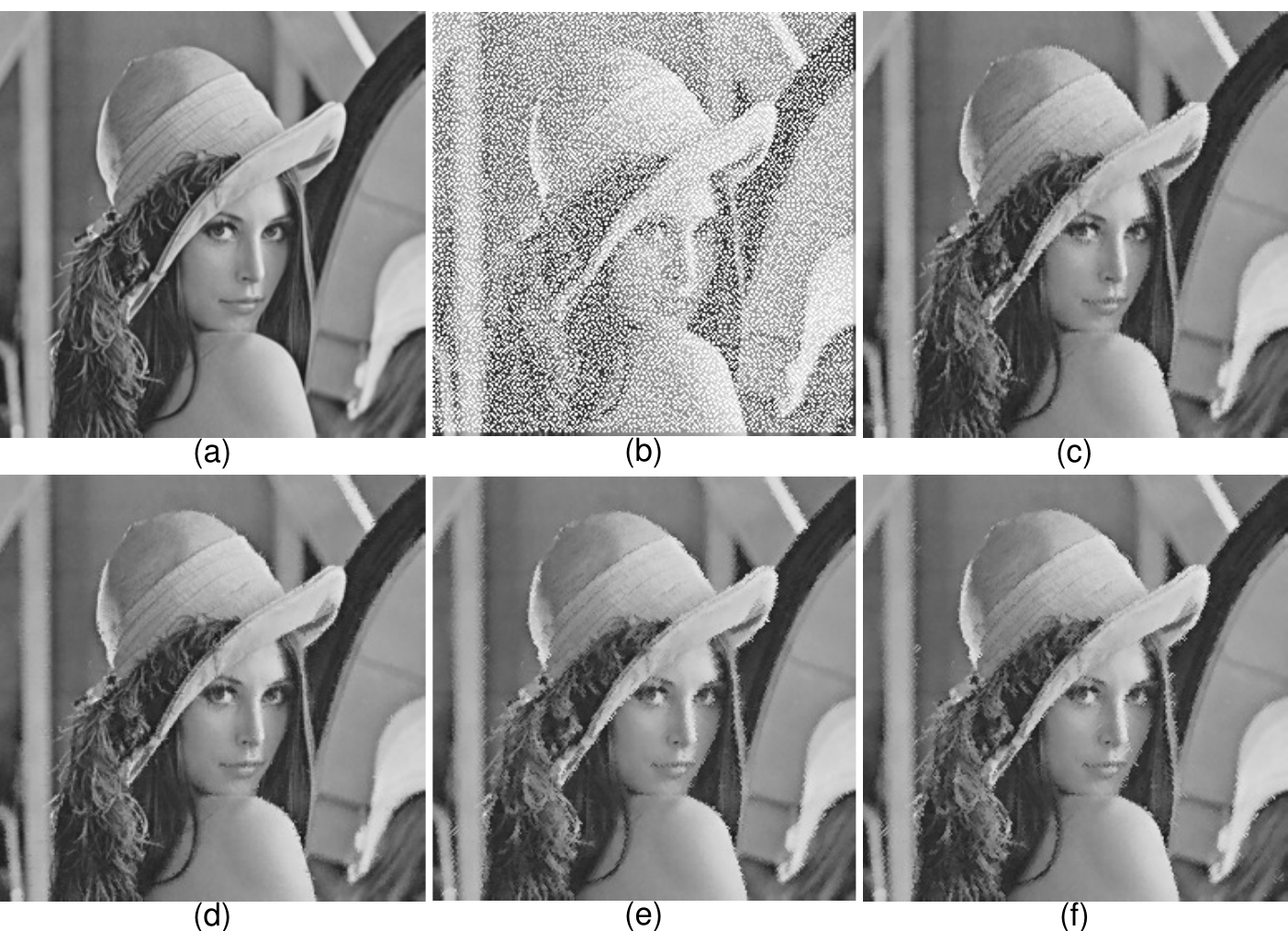}\\
	\caption{(a) Ideal original image Lena. (b) Degraded image (with $43.5\%$ pixels damaged).  (c) Reconstructed image by GN1 + CRT. (d) Reconstructed image by GN2 + CRT. (e) Reconstructed image by MED + CRT. (e) Reconstructed image by CRT + MED.}\label{FIGresimage10}
\end{figure*}

We are in position to compare the performance of interpolation method and median filtering in the problem of restoring damaged pixels.
Specifically, we consider three methods: GN1, GN2 and median filter (MED for short). In general, GN2 requires a prerequisite condition of differentiable since it  involves   derivative.  It would be stretching a point to describe  a digital image as a set of samples of a smooth (differentiable) function.   Nevertheless, the introduction of difference (also called derivative in some literature without ambiguity) for the original digital image  could help to preserve more useful information in the reconstructed image. The experimental results are shown in Figure \ref{FIGresimage10} and  Table \ref{qmforimageres}.
\begin{table}[!ht]
	\caption{Quantitative measurements for quality of image recovery results.}\label{qmforimageres}
	\centering
	\vspace{0.1cm}
	\begin{tabular}{l|c|c|c|c}	
		\hline
		& GN1 + CRT & GN2 + CRT & MED + CRT & CRT +MED \\
		\hline
		RMSE: $\delta$  & 0.0570 & \textbf{0.0488} & 0.0884  & 0.0875  \\
		\hline
		PSNR: $\rho$ &  30.56 &  \textbf{31.91} &  26.76 & 26.84 \\
		\hline
		CC : $\gamma$ & 0.9874 &  \textbf{0.9908}& 0.9719 & 0.9726 \\
		\hline
	\end{tabular}
\end{table}
Here CRT represents the correction operation.  We use relative mean square error (RMSE) $\delta$, peak signal to noise ratio (PSNR) $\rho$ and correlation coefficient (CC) $\gamma$, to measure the quality of reconstructed images. They are defined respectively as:
\begin{equation*}
	\delta(\mathcal{I},\mathcal{I}_{r}) =  \frac{\norm{\mathcal{I} -\mathcal{I}_{r}}_F}{\norm{\mathcal{I}}_F},
\end{equation*}
\begin{equation*}
	\rho(\mathcal{I},\mathcal{I}_{r}) = 10\log_{10}\left (\frac{255^2\times L_1\times L_2}{\norm{\mathcal{I} -\mathcal{I}_{r}}_F^2}\right),
\end{equation*}
\begin{equation*}
	\gamma ( \mathcal{I},\mathcal{I}_{r} )  = \frac{\sum_{i,j}(\mathcal{I}(i,j)-\mathcal{I}^0)(\mathcal{I}_{r}(i,j)-\mathcal{I}_{r}^0)}
	{\norm{\mathcal{I}-\mathcal{I}^0}_F \norm{\mathcal{I}_{r}-\mathcal{I}_{r}^0}_F},
\end{equation*}
where  $\mathcal{I} $, $\mathcal{I}_{r}$  denote  original   and reconstructed image respectively, $\mathcal{I}^0$, $\mathcal{I}_{r}^0$ denote their   averaged pixel values, and $\norm{\cdot}_F$ denotes Frobenius norm,  and $L_1$ and  $ L_2$ are the number of rows and columns of $\mathcal{I}$.

From Table \ref{qmforimageres}, we conclude that the interpolation-based method GN1 performs significantly better than the median filtering method. If there is  some information about gradient of original image available to  be utilized, the performance of   image recovery  can be improved further by  GN2. These conclusions are also reflected in Figure \ref{FIGresimage10} visually.

It is noted that we consider the image recovery problem only from the point where a digital image is degraded by simply wiping out some  pixel values. Besides,  the  material about recovery methods developed in this section is far from  exhaustive. Even so, the nonuniform-interpolation-based image recovery methods perform well and are easily implemented. It is conceivable that these methods could be integrated into some more comprehensive image recovery approaches. These further explorations, although of importance in image processing, are beyond the scope of this paper.

%=======================================================
\section{Conclusion}\label{S7}

Several interpolation formulas associated with non-uniformly distributed data are presented. If the signal to be reconstructed is bandlimited, then it is possible to reconstruct the entire signal by sampling it with the total number of samples larger than the corresponding bandwidth. For the case of non-bandlimited signal, quantitative error analysis for reconstructing is also analyzed. It has been shown that the introducing derivative samples  of function can improve reconstruction result significantly. %A fast algorithm based on FFT makes multichannel interpolation more effective and stable.
As an application, several nonuniform-interpolation-based algorithms for recovering a certain kind of corrupted images are demonstrated. The performance is satisfactory.

% Generated by IEEEtran.bst, version: 1.12 (2007/01/11)


\begin{thebibliography}{10}
	\providecommand{\url}[1]{#1}
	\csname url@samestyle\endcsname
	\providecommand{\newblock}{\relax}
	\providecommand{\bibinfo}[2]{#2}
	\providecommand{\BIBentrySTDinterwordspacing}{\spaceskip=0pt\relax}
	\providecommand{\BIBentryALTinterwordstretchfactor}{4}
	\providecommand{\BIBentryALTinterwordspacing}{\spaceskip=\fontdimen2\font plus
		\BIBentryALTinterwordstretchfactor\fontdimen3\font minus
		\fontdimen4\font\relax}
	\providecommand{\BIBforeignlanguage}[2]{{%
			\expandafter\ifx\csname l@#1\endcsname\relax
			\typeout{** WARNING: IEEEtran.bst: No hyphenation pattern has been}%
			\typeout{** loaded for the language `#1'. Using the pattern for}%
			\typeout{** the default language instead.}%
			\else
			\language=\csname l@#1\endcsname
			\fi
			#2}}
	\providecommand{\BIBdecl}{\relax}
	\BIBdecl
	
	\bibitem{zayed1993advances}
	A.~I. Zayed, \emph{Advances in {S}hannon's sampling theory}.\hskip 1em plus
	0.5em minus 0.4em\relax CRC press, 1993.
	
	\bibitem{sidky2008image}
	E.~Y. Sidky and X.~Pan, ``Image reconstruction in circular cone-beam computed
	tomography by constrained, total-variation minimization,'' \emph{Physics in
		Medicine \& Biology}, vol.~53, no.~17, p. 4777, 2008.
	
	\bibitem{liang2000principles}
	Z.-P. Liang and P.~C. Lauterbur, \emph{Principles of magnetic resonance
		imaging: a signal processing perspective}.\hskip 1em plus 0.5em minus
	0.4em\relax SPIE Optical Engineering Press, 2000.
	
	\bibitem{burke2009introduction}
	B.~F. Burke and F.~Graham-Smith, \emph{An introduction to radio
		astronomy}.\hskip 1em plus 0.5em minus 0.4em\relax Cambridge University
	Press, 2009.
	
	\bibitem{yen1956nonuniform}
	J.~Yen, ``On nonuniform sampling of bandwidth-limited signals,'' \emph{IRE
		Transactions on circuit theory}, vol.~3, no.~4, pp. 251--257, 1956.
	
	\bibitem{yao1967some}
	K.~Yao and J.~Thomas, ``On some stability and interpolatory properties of
	nonuniform sampling expansions,'' \emph{IEEE Transactions on Circuit Theory},
	vol.~14, no.~4, pp. 404--408, 1967.
	
	\bibitem{jerri1977shannon}
	A.~J. Jerri, ``The {S}hannon sampling theorem - its various extensions and
	applications: a tutorial review,'' \emph{Proceedings of the IEEE}, vol.~65,
	no.~11, pp. 1565--1596, 1977.
	
	\bibitem{feichtinger1992irregular}
	H.~G. Feichtinger and K.~Gr{\"o}chenig, ``Irregular sampling theorems and
	series expansions of band-limited functions,'' \emph{Journal of Mathematical
		Analysis and Applications}, vol. 167, no.~2, pp. 530--556, 1992.
	
	\bibitem{seip1987an}
	K.~Seip, ``An irregular sampling theorem for functions bandlimited in a
	generalized sense,'' \emph{SIAM J. Appl. Math.}, vol.~47, no.~5, pp.
	1112--1116, 1987.
	
	\bibitem{Maymon2011sinc}
	S.~Maymon and A.~V. Oppenheim, ``Sinc interpolation of nonuniform samples,''
	\emph{IEEE Trans. Signal Process.}, vol.~59, no.~10, pp. 4745--4758, Oct
	2011.
	
	\bibitem{higgins2000sampling}
	J.~Higgins, G.~Schmeisser, and J.~Voss, ``The sampling theorem and several
	equivalent results in analysis,'' \emph{J. Comput. Anal. Appl.}, vol.~2,
	no.~4, pp. 333--371, 2000.
	
	\bibitem{liu2010new}
	Y.~L. Liu, K.~I. Kou, and I.~T. Ho, ``New sampling formulae for non-bandlimited
	signals associated with linear canonical transform and nonlinear {F}ourier
	atoms,'' \emph{Signal Process.}, vol.~90, no.~3, pp. 933--945, 2010.
	
	\bibitem{cheng2017novel}
	D.~Cheng and K.~I. Kou, ``Novel sampling formulas associated with quaternionic
	prolate spheroidal wave functions,'' \emph{Adv. Appl. Clifford Algebr.},
	vol.~27, no.~4, pp. 2961--2983, 2017.
	
	\bibitem{chengkou2018generalized}
	------, ``Generalized sampling expansions associated with quaternion {F}ourier
	transform,'' \emph{Math. Meth. Appl. Sci.}, vol.~41, no.~11, pp. 4021--4032,
	2018.
	
	\bibitem{XU2016311}
	\BIBentryALTinterwordspacing
	L.~Xu, F.~Zhang, and R.~Tao, ``Randomized nonuniform sampling and
	reconstruction in fractional {F}ourier domain,'' \emph{Signal Process.}, vol.
	120, pp. 311--322, 2016. [Online]. Available:
	\url{http://www.sciencedirect.com/science/article/pii/S0165168415003187}
	\BIBentrySTDinterwordspacing
	
	\bibitem{margolis2008nonuniform}
	E.~Margolis and Y.~C. Eldar, ``Nonuniform sampling of periodic bandlimited
	signals,'' \emph{IEEE Trans. Signal Process.}, vol.~56, no.~7, pp.
	2728--2745, 2008.
	
	\bibitem{NAVASCUES2018}
	\BIBentryALTinterwordspacing
	M.~Navascués, S.~Jha, A.~Chand, and M.~Sebastián, ``Generalized trigonometric
	interpolation,'' \emph{Journal of Computational and Applied Mathematics},
	2018. [Online]. Available:
	\url{http://www.sciencedirect.com/science/article/pii/S0377042718304801}
	\BIBentrySTDinterwordspacing
	
	\bibitem{cauchy1841memoire}
	A.~Cauchy, ``Memoire sur diverses formulas d'analyse,'' \emph{Compte Rendu
		(Paris)}, vol.~12, pp. 283--298, 1841.
	
	\bibitem{benedetto2012modern}
	J.~J. Benedetto and P.~J. Ferreira, \emph{Modern sampling theory: mathematics
		and applications}.\hskip 1em plus 0.5em minus 0.4em\relax Springer Science \&
	Business Media, 2012.
	
	\bibitem{schanze1995sinc}
	T.~Schanze, ``Sinc interpolation of discrete periodic signals,'' \emph{IEEE
		Trans. Signal Process.}, vol.~43, no.~6, pp. 1502--1503, 1995.
	
	\bibitem{candocia1998comments}
	F.~Candocia and J.~C. Principe, ``Comments on "{S}inc interpolation of discrete
	periodic signals",'' \emph{IEEE Trans. Signal Process.}, vol.~46, no.~7, pp.
	2044--2047, 1998.
	
	\bibitem{dooley2000notes}
	S.~R. Dooley and A.~K. Nandi, ``Notes on the interpolation of discrete periodic
	signals using sinc function related approaches,'' \emph{IEEE Trans. Signal
		Process.}, vol.~48, no.~4, pp. 1201--1203, 2000.
	
	\bibitem{jacob2002sampling}
	M.~Jacob, T.~Blu, and M.~Unser, ``Sampling of periodic signals: A quantitative
	error analysis,'' \emph{IEEE Trans. Signal Process.}, vol.~50, no.~5, pp.
	1153--1159, 2002.
	
	\bibitem{xiao2013sampling}
	L.~Xiao and W.~Sun, ``Sampling theorems for signals periodic in the linear
	canonical transform domain,'' \emph{Opt. Commun.}, vol. 290, pp. 14--18,
	2013.
	
	\bibitem{selva2015fft}
	J.~Selva, ``{FFT} interpolation from nonuniform samples lying in a regular
	grid,'' \emph{IEEE Trans. Signal Process.}, vol.~63, no.~11, pp. 2826--2834,
	June 2015.
	
	\bibitem{chengkou2018multi1}
	D.~Cheng and K.~I. Kou, ``Multichannel interpolation for periodic signals via
	{FFT}, error analysis and image scaling,'' \emph{arXiv preprint
		arXiv:1802.10291}, 2018.
	
	\bibitem{king2009hilbert}
	F.~W. King, \emph{Hilbert transforms}.\hskip 1em plus 0.5em minus 0.4em\relax
	Cambridge University Press, 2009.
	
	\bibitem{mo2015afd}
	Y.~Mo, T.~Qian, W.~Mai, and Q.~Chen, ``The {AFD} methods to compute {H}ilbert
	transform,'' \emph{Appl. Math. Lett.}, vol.~45, pp. 18--24, 2015.
	
	\bibitem{strohmer2006fast}
	T.~Strohmer and J.~Tanner, ``Fast reconstruction methods for bandlimited
	functions from periodic nonuniform sampling,'' \emph{SIAM J. Numer. Anal.},
	vol.~44, no.~3, pp. 1073--1094, 2006.
	
	\bibitem{papoulis1977generalized}
	A.~Papoulis, ``Generalized sampling expansion,'' \emph{{IEEE} Trans. Circuits
		Syst.}, vol.~24, no.~11, pp. 652--654, 1977.
	
	\bibitem{sommen2008relationship}
	P.~Sommen and K.~Janse, ``On the relationship between uniform and recurrent
	nonuniform discrete-time sampling schemes,'' \emph{IEEE Trans. Signal
		Process.}, vol.~56, no.~10, pp. 5147--5156, 2008.
	
	\bibitem{petrou2010image}
	M.~Petrou and C.~Petrou, \emph{Image processing: the fundamentals},
	2nd~ed.\hskip 1em plus 0.5em minus 0.4em\relax Oxford: John Wiley \& Sons,
	2010.
	
	\bibitem{Gonzalez2006digitalim}
	R.~C. Gonzalez and R.~E. Woods, \emph{Digital Image Processing}, 3rd~ed.\hskip
	1em plus 0.5em minus 0.4em\relax Upper Saddle River, NJ, USA: Prentice-Hall,
	Inc., 2006.
	
\end{thebibliography}
\end{document}